\documentclass[12pt,absolute]{mymathart}
\usepackage{mymathmacros}
\usepackage{upref}
\usepackage{mathrsfs}
\newcommand{\ben}{\begin{enumerate}}
\newcommand{\een}{\end{enumerate}}

\newcommand{\bea}{\begin{eqnarray}}
\newcommand{\ba}{\begin{array}}
\newcommand{\bean}{\begin{eqnarray*}}
\newcommand{\ea}{\end{array}}
\newcommand{\eea}{\end{eqnarray}}
\newcommand{\eean}{\end{eqnarray*}}
\newcommand{\beq}{\begin{equation}}
\newcommand{\eeq}{\end{equation}}
\newcommand{\bthm}{\begin{thm}}
\newcommand{\ethm}{\end{thm}}
\newcommand{\blem}{\begin{lem}}
\newcommand{\elem}{\end{lem}}
\newcommand{\bprop}{\begin{prop}}
\newcommand{\eprop}{\end{prop}}
\newcommand{\bcor}{\begin{cor}}
\newcommand{\ecor}{\end{cor}}
\newcommand{\bdfn}{\begin{dfn}}
\newcommand{\edfn}{\end{dfn}}
\newcommand{\brem}{\begin{rem}}
\newcommand{\erem}{\end{rem}}
\newcommand{\bpf}{\begin{proof}}
\newcommand{\epf}{\end{proof}}
\newcommand{\bfact}{\begin{fact}}
\newcommand{\efact}{\end{fact}}
\newcommand{\bobs}{\begin{obs}}
\newcommand{\eobs}{\end{obs}}

\renewcommand{\dim}{\operatorname{HD}}
\newcommand{\hdim}{\HD}
\newcommand{\dense}{\operatorname{dense}}
\newcommand{\hypdim}{\dim_{\operatorname{hyp}}}
\newcommand{\MOD}{\operatorname{mod}}

\alph{enumii} \roman{enumiii}

           \def\cM{\mathcal M}        \def\cP{{\mathcal P}}

\def\N{{\mathbb N}}                \def\Z{{\mathbb Z}}      \def\R{{\mathbb R}}
\def\C{{\mathbb C}}                

\def\a{\alpha}                             
\def\De{\Delta}               \def\e{\varepsilon}          
\def\g{\gamma}                           
\def\La{\Lambda}              \def\om{\omega}           
               \def\sg{\sigma}
               
\def\ka{\kappa}

\newcommand{\lam}{\lambda}

\def\1{1\!\!1}

\def\and{\text{ and }}

        \def\diam{\text{\rm {diam}}}

\def\H{\text{{\rm H}}}     \def\HD{\text{{\rm HD}}}

  \def\dense{\text{{\rm dense}}}
\def\EV{\text{{\rm EV}}} 
          \def\P{\text{{\rm P}}}

\def\bi{\bigcap}              \def\bu{\bigcup}
\def\({\bigl(}                \def\){\bigr)}
\def\lt{\left}                \def\rt{\right}

\def\ld{\ldots}                        \def\^{\tilde}

            \def\sms{\setminus}
\def\sbt{\subset}

      \def\imp{\Rightarrow}

\def\sp{\medskip}             \def\fr{\noindent}        
\def\ov{\overline}            \def\un{\underline}
\def\ess{{\rm ess}}
\def\cl{\text{cl}}
\def\om{\omega}

\def\D{{\mathbb D}}

\newcommand{\interior}{\operatorname{int}}

\newcommand{\lP}{\underline{P}}
\newcommand{\uP}{\overline{P}}

\numberwithin{equation}{section}

\newcommand{\m}{{\tt m}}

\newcommand{\bdyit}[2]
             {{\rule{0pt}{0pt}_{\mbox{$\scriptstyle #2$}}^{\mbox{%
                   $\scriptstyle #1$}} }}

\newcommand{\K}{\mathbb{K}}

\newcommand{\Zt}{\widetilde{Z}}
\newcommand{\Zl}{\underline{Z}}

 \newtheoremstyle{numberedstyle}
   {9pt}
   {9pt}
   {\normalfont}
   {}
   {\bfseries}
   {.}
   {\newline}
   {}

\newtheorem{thm}{Theorem}[section]%
\newtheorem{lem}[thm]{Lemma}%
\newtheorem{cor}[thm]{Corollary}%
\newtheorem{prop}[thm]{Proposition}%
\newtheorem{obs}[thm]{Observation}%
\newtheorem{dfn}[thm]{Definition}%
\newtheorem{rem}[thm]{Remark}%
\newtheorem{deflem}[thm]{Definition and Lemma}%

\theoremstyle{numberedstyle}
\newtheorem{defn}[thm]{Definition}%

\newtheorem*{rep@theorem}{\rep@title}
\newcommand{\newreptheorem}[2]{%
\newenvironment{rep#1}[1]{%
\def\rep@title{#2 \ref{##1}}%
\begin{rep@theorem}}%
 {\end{rep@theorem}}}

\newreptheorem{theorem}{Theorem}
\newreptheorem{corollary}{Corollary}

\newcommand{\U}{\mathcal{U}}

\usepackage{hyperref}

\title[Non-autonomous iterated function systems]{
   Non-autonomous \\ conformal iterated function systems \\ and
    Moran-set constructions}
\author{Lasse Rempe-Gillen}
\address{Dept.\ of Mathematical Sciences, University of Liverpool,
 Liverpool L69 7ZL, UK} 
\email{l.rempe@liverpool.ac.uk}
\author{Mariusz Urba\'nski}
\address{Department of Mathematics, University of North Texas, P.O.~Box 311430,
   Denton, TX 76203-1430, USA}
\email{urbanski@unt.edu}
\thanks{The first author was supported by EPSRC Fellowship
  EP/E052851/1. The second author was supported in part by
NSF Grant DMS 1001874.} 

\newcommand{\const}{\operatorname{const}}

\begin{document} 
 \begin{abstract}
  We study non-autonomous conformal
  iterated function systems, with finite or 
  countable infinite alphabet alike. 
   These differ from the
  usual (autonomous) iterated function systems in that the contractions applied
  at each step in time are allowed to vary. 
  (In the case where all maps are affine similarities, the resulting
  system is also called a ``Moran set construction''.)

We shall show that, given a suitable restriction on the growth of the number
  of contractions used at each step, the Hausdorff dimension of the limit
  set of such a system is determined by an equation known as
  \emph{Bowen's formula}. We also give examples that show
   the optimality of our results. 

In addition, we  
  prove Bowen's formula for a class of infinite alphabet-systems and
  deal with Hausdorff and packing measures for finite systems, 
  as well as continuity of topological pressure and Hausdorff dimension for
  both finite and infinite systems. In particular we strengthen the
  existing continuity results for infinite autonomous systems. 

 As a simple application of our results, we show that, for a transcendental
  meromorphic function $f$, the Hausdorff dimension of the set of transitive
  points (i.e., those points whose orbits are dense in the Julia set)
  is bounded from below by the \emph{hyperbolic dimension} of $f$ (in the sense of Shishikura).
 \end{abstract}
 
\maketitle

\section{Introduction}

\subsection{Background}
 In the classical theory of \emph{iterated function systems},
  one is given a finite set 
  $\Phi=\{\phi_i\}_{i\in I}$ of uniformly contracting
  affine similarities $\phi_i:\R^d\to\R^d$ and studies the
  \emph{limit set} $J=J(\Phi)$ of the system $\Phi$, defined
  as the set of possible limit points of sequences
  $\phi_{\om_1}(\phi_{\om_2}(\dots(\phi_{\om_n}(x))\dots))$, with
  $x\in\R^d$ and $\om_j\in I$ for all $j$. Famous examples of
  such ``self-similar'' sets include the Middle-third Cantor set, the
  Sierpi\'nski triangle, the van Koch curve, and many others. 

The theory of systems where the contractions $\phi_i$ are not necessarily
  affine similarities but merely conformal, and where the alphabet $I$ is not
  necessarily finite but only required to be countable, is also well
  developed; see e.g.\
  \cite{mulondon} (which laid the foundation for the theory of
  infinite systems), 
  \cite{mauldinurbanskiGDMS}, and \cite{hensley}.

 From the point of view of dynamical systems, the above represent 
  an \emph{autonomous} setting: the rule applied to the system
  is independent of the time $n$. The analogy to ordinary
  dynamical systems is particularly transparent if the ranges of all
  contractions $\phi_i$, $i\in I$, are mutually disjoint. Then one can 
  consider the dynamics on $J(\Phi)$  given by the \emph{inverses}
  of the functions $\phi_i:J\to\phi_i(J)$. This gives an
  expanding non-invertible dynamical system on the limit set. Conversely,
  we may often pass from a non-invertible dynamical system $f$ to an iterated
  function system whose maps $\phi_i$ are inverse branches of iterates of $f$.

The first and often most important question regarding an autonomous conformal
  iterated function system $\Phi$ is how
  to determine the Hausdorff dimension $\HD(J(\Phi))$ 
  of the limit set $J(\Phi)$.
  In the classical setting described above, it is well-known that this 
  dimension is given by the solution $t$ to the equation
  $\sum_{i\in I}\|D\phi_i\|^t=1$. For a general autonomous conformal iterated
   function system, the number $t$ generalizes to the \emph{Bowen dimension}
   $B(\Phi)$. Informally speaking, $B(\Phi)$ is the number that
   is obtained by modifying the usual definition of Hausdorff dimension to
   allow only the natural coverings of $J(\Phi)$ corresponding to the
   $n$-th levels in the iterative construction. Then
   \emph{Bowen's formula} states that $\HD(J(\Phi))= B(\Phi)$.
    The pioneer work here belongs to Moran \cite{moran} and Bowen
    \cite{bowen}, while Bowen's formula was proved     
    in \cite{mauldinurbanskiGDMS} for general
    infinite autonomous 
    conformal iterated autonomous iterated function systems (and even more
    generally ``graph directed Markov systems''). 

\subsection*{Non-autonomous systems}
 It is natural, and frequently necessary in applications, to consider the
 \emph{non-autonomous} version of the above setting, where the system
 $\Phi$ is allowed to vary with time $n$. More precisely, 
 the system $\Phi$ consists of a sequence $(\Phi^{(n)})_{n\in\N}$ of
  collections of
  conformal contractions, where $\Phi^{(n)}=\{\phi_i^{(n)}\}_{i\in I^{(n)}}$ may
  vary with $n$. The \emph{limit set} $J(\Phi)$ now consists of the
  possible limits of 
  sequences $\phi_{\om1}^{(1)}(\phi_{\om_2}^{(2)}( \dots \circ
  \phi_{\om_n}^{(n)}(x))\dots))$, where 
  $\om_j\in I^{(j)}$ for all $j$. 

 In the case where all $\phi_i^{(n)}$ are affine similarities,
  this is also referred to as
  a \emph{Moran set construction}.\footnote{%
 Moran set constructions
  are somewhat more general than our non-autonomous systems, in that
  the placement of pieces at
  level $n+1$
  is allowed to occur independently for each level $n$ piece. All our results
  (and proofs) will remain true in this somewhat more general set-up. However,
  we feel that our setting is natural from the point of view of 
  dynamics, and prefer not to increase the notational complexity further.}
  This
   entire paper is devoted to study conformal non-autonomous systems. However,
  we also touch on autonomous systems: On the one hand, they occur as 
  auxiliary objects in our non-autonomous constructions, on the other hand,
   we also prove new meaningful theorems about
  autonomous systems.

Our central objective is to establish versions of Bowen's formula that
  are as general as possible. But we will also touch on other
  aspects, such as the continuity of Hausdorff dimension and
  the relationship of nonautonomous systems to \emph{random} iterated
  function systems
  as introduced in \cite{royurblambda}.

\subsection*{Bowen's Formula.} The definition of the Bowen dimension
$B(\Phi)$ generalizes naturally to 
  non-autonomous systems (see Definition~\ref{defn:bowen}). However,
  in this setting 
   Bowen's formula $\HD(J(\Phi))=B(\Phi)$ no longer holds in general.
   (This follows from our theorems below, but has been known for a long time; 
    compare e.g.\ \cite{moransetsclasses}.) Our key observation for finite systems
    (i.e., those for which all index sets $I^{(n)}$ are finite) is that Bowen's formula
    \emph{does} hold when we restrict the growth of the size $\# I^{(n)}$ of the $n$-th level
    index sets. Throughout the article, the dimension of the ambient space is denoted by $d$. 

\begin{thm}[Bowen's formula for systems of sub-exponential
  growth]\label{thm:bowensubexponential}
Suppose that $\Phi$ is a non-autonomous conformal iterated function system
of \emph{sub-exponential growth}:
\[ 
\lim_{n\to\infty} \frac1n\log \# I^{(n)}= 0.  
\] 
Then Bowen's Formula holds for $\Phi$, i.e. $\HD(J(\Phi))=B(\Phi)$.

On the other hand, suppose that $0<t_1<t_2<d$ and let $\eps>0$. Then there exists a non-autonomous conformal iterated function
   system $\Phi$ such that 
  \[ \limsup_{n\to\infty} \frac1n \log \# I^{(n)} \leq \eps, \]
   $\HD(J(\Phi))=t_1$ and $B(\Phi)=t_2$.
\end{thm}

The counterexamples we construct to prove the second part of the theorem
  are very irregular: there are many stages consisting of pieces
  of definite size, but at some stages the number of pieces is large. Also, the examples cannot be
  extended to the case $t_1=0$ or $t_2=d$. The following theorem shows that both of these restrictions are necessary.

 \begin{thm}[Bowen's formula for some systems of exponential
   growth]\label{thm:bowenexponential} 
Suppose that $\Phi$ is a non-autonomous conformal iterated function system
of at most \emph{exponential growth}:
\[ 
\limsup_{n\to\infty} \frac1n\log \# I^{(n)} < \infty.  
\] 
Suppose additionally that one of the following is true:
\begin{enumerate}
   \item $\|D\phi_i^{(n)}\|\to 0$ uniformly in $i$ as $n\to\infty$, \label{item:smallpieces}
   \item $\HD(J(\Phi))=0$, or   \label{item:extremalHD}
   \item $B(\Phi)=d$.  \label{item:extremalB}
\end{enumerate}
Then Bowen's Formula holds for $\Phi$, i.e. $\HD(J(\Phi))=B(\Phi)$.

On the other hand, let $(\alpha_n)_{n\in\N}$ be an arbitrary sequence of positive integers such that 
    $\lim_{n\to\infty} \log\alpha_n/n = \infty$. Then
    there exists a non-autonomous conformal iterated function
   system $\Phi$ such that $\# I^{(n)} \leq \alpha_n$ for all $n$, and such that
   (\ref{item:smallpieces}) to (\ref{item:extremalB}) hold for $\Phi$.
\end{thm}

 The idea of proof for both of the preceding theorems is as follows. We prove a general lower bound on 
   $\HD(\Phi)$ for a finite non-autonomous conformal iterated function system, and use this bound to 
   establish the above results assuming additional ``balancing'' assumptions (meaning that we restrict
   how much the pieces at a given level may differ in size). We are then able to remove those assumptions
   by studying suitable subsystems (this is an idea that stems from the study of infinite iterated function systems;
   the novelty of our approach lies in its application to finite cases). 

 A special case of Theorem \ref{thm:bowenexponential} concerns systems
  of sufficiently regular 
   exponential growth:

\begin{prop}[Regular exponential growth]\label{prop:bowenregular}
Suppose that $\Phi$ a system such that both limits
    \[ a := \lim_{n\to\infty}\frac1n\log\# I^{(n)} \]
and
\[
b:=\lim_{n\to\infty,j\in I^{(n)}}\frac1n\log\bigl(1/\|D\phi_j^{(n)}\|\bigr)
\]
exist and are finite and positive. Then $B(\Phi)=a/b$, and hence $\hdim(J(\Phi))=B(\Phi)=a/b$. 
\end{prop}

As an application of our results, let $\lambda>1$ and consider the set
of real numbers $x\in [0,1]$ whose continued fraction 
   expansion $x=[a_0,a_1,a_2,\dots]$ satisfies 
   $\lambda^n < a_n < \lambda^{n+1}$ for sufficiently large $n$. 
  It
   follows from a result of Jordan and Rams 
   \cite{jordanrams} that this set has Hausdorff dimension $1/2$. We
   can also see this directly from the above Proposition; indeed, here
   we have 
      \[ a = \lim_{n\to\infty} \frac{1}{n} \log
      (\lambda^{n+1}-\lambda^n) = \log \lambda\] 
    and
     \[ b = \lim_{n\to\infty}\frac{1}{n} \log \lambda^{2(n+1)} = 2\log
     \lambda. 
\] 
Having established Bowen's formula for finite systems we 
    also prove a version for a large class sufficiently regular
    infinite ones. This is provided by Corollary~\ref{BowenM}
    in  
    Section~\ref{scia} (Systems with Countably Infinite Alphabet).
   We also briefly comment on the relation of our results to the study 
   of random iterated function systems, as well as the continuity
   of pressure and Hausdorff dimension
   (following the line of research initiated in 
   \cite{royurblambda} 
    and continued in \cite{royurb2} 
   in the context of
    autonomous systems).
    The latter is done under
    considerably weaker hypotheses than in \cite{royurblambda} giving
    stronger results even for autonomous systems. 

 In an appendix, we provide a simple and useful application of our methods
   to complex dynamics:
 \begin{thm}[Transitive points of meromorphic functions]\label{thm:meromorphic}
  Let $f:\C\to\Ch$ be a non-linear, non-constant meromorphic function, and let 
   $J_{\dense}$ denote the set of \emph{transitive points}. That is, 
   $J_{\dense}$ consists of those points in the Julia set 
   $J(f)$ whose orbit is dense in $J(f)$. 

  Then $\HD(J_{\dense}(f)) \geq \hypdim(f)$, where $\hypdim(f)$ denotes the
   \emph{hyperbolic dimension} of $f$ in the sense of Shishikura.
 \end{thm}
 (We also prove a version of this result for autonomous infinite iterated
   function systems.)

\subsection*{Previous and related results}
 Non-autonomous systems and Moran-set constructions have been studied by a variety of 
  authors previously, and we cannot give a survey here; instead we 
  shall mention a small number of papers that are particularly relevant to our
  study. 

The article \cite{moransetsclasses} from 2001 presents 
  a survey of results known at the time. Zhi-Ying Wen has kindly pointed out
  that, in the case where $d=1$ and all maps $\phi_i^{(n)}$ are affine similarities, a slightly
 stronger result than the positive part of
 Theorem \ref{thm:bowensubexponential} appears in 
 \cite[Theorem 2]{four_authors}. However, the proof appears to use the one-dimensionality
 of the phase space in an essential way.

We would also like to mention recent results independently obtained by
  Holland and Zhang \cite{holland}. Similarly to us,
  they consider a type of Moran set construction in $\R^d$
  and establish Bowen's formula (as well as formulae
  for the packing dimension) in certain cases. However, while there is some
  overlap
  between the results,  \cite{holland} takes a different route from the present article. 
  In particular, the results there 
  assume some control over the minimal contraction factors, and in particular 
  do not imply our theorems stated in the introduction.

\subsection*{Structure of the paper}
  Section~\ref{sec:definitions} contains the definitions that are fundamental
   for our study: non-autonomous systems, words, limit sets, lower pressure and
   Bowen dimension. In Section \ref{LBHD} we establish a general
   lower bound on the Hausdorff dimension of the limit set of a finite
   non-autonomous conformal iterated functions system, which is then used
   in Section~\ref{sec:balance} to prove preliminary versions of Theorems~
   \ref{thm:bowensubexponential} and \ref{thm:bowenexponential} above. 
   The proof of 
   both theorems is completed in Sections~\ref{sec:approximation} and
   \ref{sec:counterexamples} by considering suitable subsystems for the
   positive direction, and by
   discussing the construction of 
   counterexamples to Bowen's formula that establish the optimality of our 
   results. Completing our treatment of finite systems, we briefly discuss
   Hausdorff and packing measures of limit sets for a class of uniformly
   finite systems in Section \ref{hpm}.

  In Section~\ref{scia}, we use the approximation by subsystems developed
   in Section~\ref{sec:approximation} to prove Bowen's formula for
   some large classes 
   of infinite systems. Random iterated function systems are discussed
   in Section~\ref{sec:rcifs}, 
   and continuity of pressure and Hausdorff dimension
   is treated in Section~\ref{continuity}.

 We discuss the above-mentioned application to transitive points
   in complex dynamics and
   in infinite iterated function systems in an Appendix. 

\subsection*{Acknowledgements}
We would like to thank Mark Holland, Thomas Jordan and Zhi-Ying Wen for interesting discussions
  about this work. We are also grateful to the referee for their careful reading of the paper
  and a multitude of corrections and helpful
  suggestions.

\section{Definitions and First Estimates} \label{sec:definitions}

 In this section, we define the class of non-autonomous conformal
 iterated function systems, which we shall abbreviate as NCIFS, in a
 Euclidean space $\R^d$. Our results will be stated and proved in this setting.
 We essentially follow the notation for infinite iterated function systems
 in \cite{mauldinurbanskiGDMS}. We also define the upper and lower pressure
  functions for such systems, and derive some of their basic general
  properties. To conclude the section, 
  we establish an elementary upper bound for the Hausdorff
  dimension of the limit set of any non-autonomous conformal
 iterated function system. This is the easier part of Bowen's formula.

Throughout the article, we fix $d\in\N$ and
a compact set $X\subset\R^d$ with $\overline{\interior(X)}=X$, with
the additional geometric assumption that
  $\partial X$ is smooth or that $X$ is convex. (More generally, we can allow
  sets that satisfy the more technical condition (2.7) from
  \cite{mulondon}.) 

Given a conformal map $\phi:X\to X$ we denote by $\phi_i'(x)$ or by
$D\phi_i(x)$ the derivative of $\phi$ evaluated at $x$,
i.e. $\phi_i'(x):\R^d\to\R^d$ is a similarity linear map, and we
denote by $|\phi_i'(x)|$ (or by $|D\phi_i(x)|$) its scaling factor. We
also put
$$
||D\phi||=||\phi'||=\sup\{|\phi'(x)|:x\in X\}.
$$
  
 \begin{defn} \label{defn:ncifs}
  A \emph{non-autonomous conformal iterated function system (NCIFS)} $\Phi$ on the set $X$ 
   is
   given by a sequence $\Phi^{(1)},\Phi^{(2)},\Phi^{(3)},\dots$, where each
   $\Phi^{(j)}$ is 
   a set of functions $(\phi^{(j)}_i:X\to X)_{i\in I^{(j)}}$ and each
   $I^{(j)}$ is a (finite or countably infinite) 
   index set, if the following hold. 
  \begin{enumerate}
   \item \emph{Open set condition}: We have
$$
\phi^{(j)}_a(\interior(X))\cap\phi^{(j)}_b(\interior(X))=\emptyset
$$ 
for all $j\in \N$ and all distinct indices $a,b\in I^{(j)}$. 
  
\sp \item \emph{Conformality}: There exists an open connected set
  $V\supset X$ (independent of $i$ and $j$) 
   such that each $\phi^{(j)}_i$ extends to a $C^1$
   conformal diffeomorphism of $V$ into $V$.

\sp \item \emph{Bounded distortion}: There exists a constant $K\ge 1$ such that,
    for any $k\leq l$ and any $\om_k,\om_{k+1},\dots,\om_l$ with
    $\om_j\in I^{(j)}$, the map $\phi := \phi_{\om_k}\circ\dots\circ
    \phi_{\om_l}$ satisfies 
$$
\|D\phi(x)\|\leq K\|D\phi(y)\|
$$ 
for all $x,y\in V$. 

\sp \item \emph{Uniform contraction}: There is a constant $\eta<1$ such that 
$$
\|D\phi(x)\|\leq \eta^m
$$ 
for all sufficiently large $m$, all $x\in X$ and all
    $\phi = \phi_{\om_{j}}\circ\dots\circ \phi_{\om_{j+m}}$, where
     $j\ge 1$ and $\om_k \in I^{(k)}$. In particular, this holds if 
$$
\|D\phi_i^{(j)}(x)\|\leq \eta
$$ 
for all $j\ge 1$ and all $x\in X$, which we assume in the sequel for
the ease of exposition. 
\end{enumerate}

The system $\Phi$ is called
\emph{autonomous} if $I^{(n)}$ and $\Phi^{(n)}$ are independent of $n$.
\end{defn}

\begin{remark}
 We 
  remark that condition (c) (bounded distortion) is automatically
  satisfied when $d\geq 2$. Indeed, for $d\ge 3$ this
  condition can be deduced from the celebrated Liouville Theorem
  asserting that 
  each conformal map is a composition of an inversion with respect
  to some sphere (perhaps of infinite radius), a Euclidean linear
  similarity and a translation. In the case when $d=2$ any conformal
  map is either holomorphic or anti-holomorphic and condition (c) can
  be easily deduced from the celebrated Koebe Distortion
  Theorem. 

  Hence condition (c) only needs to be 
  verified in the case $d=1$. Of course bounded distortion is always
  satisfied (with $K=1$) if the system consists of similarities. 
\end{remark}

 For the definitions and discussions that follow, we fix some
  nonautonomous conformal IFS $\Phi$. The combinatorial
  language we introduce will be used throughout the article.

 \begin{defn}[Words]
  We define symbolic spaces, for $0< m \leq n <\infty$ by
   \[
    I^{n} := \prod_{j=1}^n I^{(j)}, \quad I^{\infty} := \prod_{j=1}^{\infty} I^{(j)},
      \quad
      I^{m,n} := \prod_{j=m}^n I^{(j)}\quad\text{and}\quad 
               I^{m,\infty} := \prod_{j=m}^{\infty} I^{(m)}. \]
  Elements of $I^{n}$ (with $n\leq \infty$) are called 
    \emph{(initial) words}, while those of
    $I^{m,n}$ with $m>1$ are called \emph{non-initial words}. We assume all words
    are initial, unless explicitly stated otherwise.
    
  The \emph{length} of a word $\omega\in I^{m,n}$ is $|\omega| := n-m$;
   $\omega$ is called 
    \emph{finite} or \emph{infinite} according to whether its length is finite
    or infinite. Finite words are sometimes also referred to as \emph{blocks}.

 If $\om=\om_m \om_{m+1}\dots \om_n\in I^{m,n}$ is a finite word, 
   we define the associated conformal
   map by
\[  \phi^{m,n}_{\om}:= \phi^{(m)}_{\om_{m}}\circ\dots\circ \phi^{(n)}_{\om_{n}}. \]
   In the case of an initial word, where $m=1$, we also abbreviate
     $\phi_{\om} := \phi_{\om}^n:= \phi_{\om}^{1,n}$. 
 \end{defn}
\begin{remark}
  We note the similarity of notation between the individual index sets
   $I^{(n)}$ and the $n$-fold product $I^n$. No confusion should arise,
   as it will always be clear from context whether we mean symbols or words. 
\end{remark}

We can now define the central object of study:
  the \emph{limit set} of a nonautonomous IFS.
\begin{defn}[Limit set]
 For all $n\in\N$ and $\om\in I^{n}$, we define
\[     X_{\om} := \phi^n_{\om}(X) \quad\text{and}\quad
     X_n :=\bigcup_{\omega\in I^{n}} X_{\om}. \]
  The \emph{limit set} (or attractor) of $\Phi$ is defined as 
\[ J := J(\Phi) := \bigcap_{n=1}^{\infty} X_n. \] 
\end{defn}
\begin{remark}
 In the case where all sets $I^{(j)}$, $j\ge 1$, are finite, the
  limit set $J(\Phi)$ is compact, and hence closed,
   as an intersection of compact sets. 
   If some of the sets $I^{(j)}$ are infinite, the
   limit set $J(\Phi)$ is not closed in general, and its closure can be
   much bigger than $J(\Phi)$---both topologically and
    with respect to Hausdorff
   dimension. Indeed, it is not difficult to construct infinite
   autonomous
   systems such that $J(\Phi)$ is dense in $X$ but has Hausdorff dimension
   equal to zero.
\end{remark}

There is another useful description of the limit set. If
$\om\in I^{m,n}$ (where $n$ is either
finite or infinite) and $s\le |\om|$, then by $\om|_s$ we mean the word
$$
\om_{m}\om_{m+1}\dots\om_{m+s-1} \in I^{m,m+s-1},
$$
i.e. the initial subword of $\om$ of length $s$. If $\om\in I^{\infty}$, 
 then $\(\phi_{\om|_n}(X)\)_{n=1}^\infty$ is
 a descending sequence of compact sets whose diameters converge to zero
 exponentially fast. Indeed,
  the uniform contraction assumption means that
\beq \label{eqn:contraction}
 \|\phi_{\om}'\| \leq \eta^n
\eeq
 and hence 
\beq\label{na120111104}
\diam\(\phi_{\om}(X)\)\le \const\eta^n
\eeq
for sufficiently large $n$ and all $\om\in I^n$.

\begin{deflem}[The projection map]\label{deflem:projection}
 Define   
\[
\pi_\Phi:I^{\infty}\to J(\Phi); \quad
  \{\pi_{\Phi}(\om)\} :=\bi_{n=1}^\infty\phi_{\om|_n}(X). \]
 Then 
$$
J(\Phi)=\pi_\Phi\left(I^{\infty}\right).
$$
\end{deflem}
\begin{proof}
  As noted above, the diameter of $\phi_{\om|_n}(X)$ tends to zero as
   $n\to\infty$, hence $\pi_\Phi$ is indeed a well-defined function,
   whose values belong to $J(\Phi)$ by definition. 

  To prove the converse, let $x\in J(\Phi)$. We use the fact that, 
   for every $n\in\N$, the set
      \[ \{\omega \in I^n: x\in X_{\omega}\} \]
   is finite. This is trivial when each index set is finite, but otherwise
   it uses the open set condition together with the fact that~-- due to the
   geometric assumptions on the boundary of $X$~-- each set
    $X_{\omega}$ with $x\in X_{\omega}$ contains a cone of definite 
    opening angle based at $x$ (see \cite[Formula (2.10)]{mulondon}). 

  Now form a directed graph on the set of finite initial words $\omega$
    with $x\in X_{\omega}$, by drawing a directed edge from $\omega^1$ to
    $\omega^2$ if and only if $\omega^1$ is obtained from  $\omega^2$ by
    deleting the last symbol. By assumption,
    this graph contains arbitrarily long directed paths that start at the
    empty word, and by the above, the degree of each vertex is finite.
    Hence, by K\"onig's lemma from graph theory, there is an infinite
    word $\omega$ such that $x\in X_{\omega|_n}$ for all $n\geq 1$. 
\end{proof}
\begin{remark}
In view of \eqref{na120111104}, the map $\pi_{\Phi}$
 is H\"older continuous if we endow $I^{\infty}$
 with any of the standard metrics $\rho_\a$, $\a>0$, where
$$
\rho_\a(\om,\tau)=\exp\(-a\cdot |\om\wedge\tau|\)
$$
and $\om\wedge\tau$ is the longest common initial segent of both $\om$
and $\tau$. 
\end{remark}

\begin{defn}[Upper and lower pressure]
 For any $t\ge 0$ and $n\in\N$, we define 
  \[ 
  Z_n(t) := Z_n^{\Phi}(t) := \sum_{w\in I^n} \|D\phi^n_{\om}\|^t, 
\]
 We now define the
\emph{upper and lower pressure functions} respectively as follows: 
   \begin{align*}
     \lP(t) &:= \lP^{\Phi}(t) := \liminf_{n\to\infty} \frac1n\log Z_n \quad\text{and}\\
     \uP(t) &:= \uP^{\Phi}(t) := \limsup_{n\to\infty}\frac1n\log Z_n.
   \end{align*} 
 Note that $\lP(t), \uP(t) \in [-\infty,+\infty]$. 
\end{defn}

 \begin{lem}[Monotonicity] \label{lem:monotonicity}
  The lower pressure function is strictly decreasing when it is finite. That is,
   if $t_1<t_2$, then either both $\lP(t_1)$ and $\lP(t_2)$ 
  are equal
   to $+\infty$, both are 
   equal to $-\infty$, or $\lP(t_1)>\lP(t_2)$. The same is true of the
   upper pressure function. 
 \end{lem}
 \begin{proof}
  Let $t\geq 0$ and $\eps>0$. Then  
    \[ Z_n(t+\eps) =
          \sum_{w\in I^n}\| D\phi^n_{\om}\|^{t+\eps} \leq
           \sum_{w\in I^n} \eta^{n\eps} \cdot \|D\phi^n_{\om}\|^t 
            =\eta^{n\eps}\cdot Z_n(t), \]
   where $\eta<1$ is the uniform contraction constant. 
    So $\lP(t+\eps)\leq \lP(t)-\eps\cdot \log(1/\eta)$. Hence,
    $\lP(t+\eps)\leq \lP(t)$, and the inequality is strict if 
    $\lP(t)$ is finite.
    
 The proof
   for the upper pressure function $\uP$ is analogous. 
 \end{proof}

We are now ready to define the \emph{Bowen dimension}, our candidate for the
  Hausdorff dimension of $J(\Phi)$.

\begin{defn}[Bowen dimension and Bowen's formula]\label{defn:bowen}
 Let $\Phi$ be a nonautonomous conformal iterated function system.
   We define
  \[ h:=h(\Phi)=\HD(J(\Phi)). \]

  The quantity
  \begin{align*}
   B := B(\Phi) &:= \sup\{t\geq 0: \lP(t)> 0\} = \inf\{t\geq 0: \lP(t)< 0 \} \\
    &= \sup\{t\geq 0:Z_n(t)\to \infty \}.
  \end{align*}
 is called  the
 \emph{Bowen dimension} of the system $\Phi$. If the equality
    \begin{equation}
       h(\Phi) =B(\phi)
    \end{equation}
  is true, we say that \emph{Bowen's formula} holds for $\Phi$.
\end{defn}
 \begin{remark}
   The equalities stated in the definition of $B(\Phi)$
  all follow from Lemma
 \ref{lem:monotonicity}. Note that 
   in the rather uninteresting case where $\lP(0)=0$, the value of the 
   two suprema  
   should be taken to be zero by definition. Also note 
   that $\lP(t)<0$ for all $t>d$, as a consequence of the open set
   condition. Hence $B(\Phi)\leq d$.  
 \end{remark}

 The Bowen dimension is precisely the value that we obtain when restricting the
  covers in the definition of Hausdorff dimension to the ``natural'' covers
  \[ \mathcal{V}_n := \{X_{\omega}: \omega \in I^n\}. \]
 In particular, 
  it is elementary that half of Bowen's formula is always satisfied:

 \begin{lem}\label{lem:HDupperBound}
  $h(\Phi) \leq B(\Phi)$.
 \end{lem}

 We shall state and prove a more general technical statement 
  (without increasing the
  difficulty of the proof), which will come for the discussion of 
  counterexamples to Bowen's formula.

 \begin{lem}\label{lem:HDuppergeneral}
   Let $\Phi$ be a NCIFS. For each $n\geq 0$, let $\mathcal{U}_n$ be a
    covering
    of the limit set of the system $\Psi_n$ defined by 
    $\Psi_n^{(j)}:=\Phi^{(n+j-1)}$, and denote its Hausdorff sum by
      \[ S(\mathcal{U}_n,t) := \sum_{U\in \mathcal{U}_n} \diam(U)^t. \]
    Suppose that $t\geq 0$ is such that
     $\liminf_{n\to\infty} Z_{n-1}(t)\cdot S(\mathcal{U}_{n},t)<\infty$. 

   Then $h(\Phi)\leq t$.
 \end{lem}
\begin{remark}
  In our applications, $\mathcal{U}_n$ will usually be a covering
   of $\bigcup_{a\in I^{(n)}} \phi_a^{(n)}(X)\supset J(\Psi_n)$. 
\end{remark}
\begin{proof}
  Consider the covering
    \[ 
     \mathcal{V}_n := \{\phi_{\om}(U):\om\in I^{n-1}, U\in \mathcal{U}_n\} 
    \]
  of $J(\Phi)$. Note that the diameters of the sets in $\mathcal{V}_n$ tend
   to zero as $n\to\infty$ by the uniform contraction assumption on $\Phi$.

  Each set in this covering
   satisfies
    \[ 
        \diam(\phi_{\om}(U)) \leq C\cdot \|D\phi_{\om}\|\cdot \diam(U), 
    \]
   where $C$ is a constant depending only on the geometry of $X$
   \footnote{If $X$ is convex, then
     $C=1$.}.
   It follows that
    \[ S(\mathcal{V}_n,t) \leq 
          C\cdot\left(\sum_{\om\in I^{n-1}}\|D\phi_{\om}\|^t\right)\cdot
             \left(\sum_{U\in\mathcal{U}_n}\diam(U)^t\right) =
         C\cdot Z_{n-1}(t) \cdot S(\mathcal{U}_n,t). \]
   Hence the assumption implies that the
    $t$-dimensional Hausdorff measure of $J(\Phi)$ is finite. 
    In particular, $h(\Phi)\leq t$, as claimed.
\end{proof}
 
 \begin{proof}[Proof of Lemma \ref{lem:HDupperBound}]
   Let $t>B(\Phi)$, and let $\U_n=\{X\}$ for all $n$. We have
     \[ \liminf_{n\to\infty}  Z_n(t)\cdot S(\U_{n+1},t) = 
          \diam(X)^t\cdot \liminf_{n\to\infty} Z_n(t) = 0 \]
   by definition of $B(\Phi)$. Hence $h(\Phi)\leq t$ by the preceding lemma.
   Since $t>B(\Phi)$ was arbitrary, the claim follows.
 \end{proof}

 To conclude the section, we remark that one 
   cannot expect Bowen's formula $B(\Phi) = J(\Phi)$ to hold without
  imposing any conditions on the non-autonomous system. Indeed, suppose that
  we are given a system $\Phi$ for which the upper pressure is strictly positive
  on the interval $[0,d)$, while the lower pressure is strictly negative on
  $(0,d]$. 

(For example, let $\Phi_1$ be an autonomous system whose limit set has Hausdorff dimension $d$, such as the 
  subdivision of the unit cube in $\R^d$ into $2^d$ cubes of side-length $1/2$, and let $\Phi_2$ be an autonomous system
  consisting of a single contraction, and whose limit set is hence a single point. If $0=N_0<N_1<\dots $ is a sufficiently rapidly increasing sequence, then 
   the system 
   \[ \Phi^{(k)} := \begin{cases}
                                \Phi_1  & \text{if } N_{2i} < k \leq N_{2i+1} \text{ for some $i\geq 0$} \\
                                \Phi_2 & \text{if } N_{2i+1} < k \leq N_{2i+2} \text{ for some $i\geq 0$} \end{cases}
     \]
    has the desired properties.)

  By Lemma \ref{lem:HDupperBound}, 
   the limit set $J(\Phi)$ has Hausdorff dimension zero. 
   Define $t_k := d - 1/k$; by assumption on the upper pressure,
   there is an increasing sequence $n_k$ such that
   \[ \liminf_{k\to\infty} \frac{1}{n_k}\log Z_{n_k}(d-1/k)\geq 0. \]
  We define a new system $\Psi$ by ``collapsing'' the levels between
  consecutive $n_k$. More precisely, $\Psi$ is given by
  index sets
    $J^{(k)} := I^{n_k,n_{k+1}-1}$ and mappings
    \[ \psi^{(k)}_{\om} := \phi^{n_k,n_{k+1}-1}_{\om} \quad
          (\omega\in J^{(k)}). \]
  Assuming without loss of generality that $n_1=1$, 
  this system has the same limit set
  $J(\Psi)=J(\Phi)$; so $h(\Psi)=h(\Phi)=0$. 
  However, the lower pressure $\lP^{\Psi}(t)$ is positive for
  every $t<d$, and hence $B(\Psi)=d$. 
  Further counterexamples to Bowen's formula will be studied in Section
  \ref{sec:counterexamples}.  

\section{A lower bound for finite systems}\label{LBHD}

In this section, we will prove a general lower bound on the Hausdorff
dimension of $J(\Phi)$ for finite systems. Let us
 begin by making the following  definition. 

\begin{defn} \label{defn:finiteness}
 A nonautonomous iterated function system is called \emph{finite} if each
alphabet $I^{(n)}$, $n\in\N$, is finite. It is called \emph{uniformly
finite} if there is a number $q\in\N$ such that $\# I^{(n)}\leq q$ for
all $n\in\N$, and 
  \emph{subexponentially bounded} if 
\[ 
\lim_{n\to\infty} \frac1n\log\# I^{(n)} = 0. 
\]
\end{defn}
Bowen's formula is based on the notion that the coverings of $J(\Phi)$ by sets of level
  $j$ (letting $j\to\infty$) provide the best coverings in terms of Hausdorff measure. As noted
  at the end of the previous section, we cannot expect this to be true without imposing
  extra conditions. One of the reasons for this is that, at some stages, the images of the maps in $\Phi^{(n)}$ might be 
  grouped in such a way that a much more efficient covering is possible---imagine, in the one-dimensional case, that there is a very
  large number of tiny intervals, all of which are grouped together so that they can be covered by a single small interval. 
By the open set condition, the volume of 
     \[ \bigcup_{a\in I^{(n)}} \phi_a^{(n)}(X) \]
 is (up to a constant factor) at least 
    \[ \# I^{(n)}\cdot \min_{a\in I^{(n)}} \|D\phi_a^{(n)}\|^d. \]
 So, if we set
  \begin{equation} \label{eqn:defunc}
    \un c_n := \min_{a\in I^{(n)}} \|D\phi_a^{(n)}\|.    
  \end{equation}
  we have 
   \[ \diam\left(\bigcup_{a\in I^{(n)}} \phi_a^{(n)}(X)\right) \geq
        \const\cdot 
  \sqrt[d]{\#I^{(n)}} \cdot {\un c}_n \cdot \diam(X). 
   \] 
 In order to obtain a lower bound on the Hausdorff dimension, 
   it thus makes sense to replace $Z_n(t)$ by 
 \begin{equation} \label{eqn:Ztilde}
\tilde Z_n(t) 
:= Z_{n-1}(t) \cdot (\# I^{(n)})^{\frac{t}{d}}\cdot 
               {\un c}_n^t
\end{equation}
 in the definition of Bowen dimension.

  We might expect that $\hdim(J(\Phi))\geq t$ provided that $\liminf
  \Zl_n(t) > 0$. The following result shows that this is indeed the
  case, provided that different pieces at the same level do not have
  drastically different sizes, as measured by the quantity
 \begin{equation} \label{eqn:defnrho}
  \rho_n := \sup_{a,b\in I^{(n)}}
        \frac{\|D{\phi^{(n)}_a}\|}{\|D{\phi^{(n)}_b}\|} 
\geq 1. 
 \end{equation}

\begin{thm}[Lower bounds on Hausdorff dimension]\label{thm:lowerbound}
  Let $\Phi$ be a finite nonautonomous conformal iterated function
  system. If $t\geq 0$ is such that 
 \[ 
\liminf_{n \to\infty} \frac{\tilde Z_n(t)}{1+\log \max_{j\leq
    n}\rho_{j}} > 0, 
\]
then $\hdim(J(\Phi))\geq t$.
\end{thm}
 
\begin{proof}
 The proof will use the \emph{mass distribution principle}, also
  known as the \emph{inverse Frostman lemma}, compare
   \cite[Section 4.1]{falconerfractal}. We begin by defining a
   sequence of probability measures 
   $(m_n)_{n\in\N}$, with $m_n$ supported on the set $X_n$, as follows.
   For any $\omega \in I^{n}$, the restriction of $m_n$ to
  $X_{\om}$ is a constant multiple of Lebesgue measure, chosen such that
\[
m_n(X_{\om})=\frac{\|D\phi_{\om}\|^t}{Z_n(t)}. 
\]
  Note that $\partial X$ has zero Lebesgue measure by the  
  geometric assumption on $X$, and hence the set of points in $X_n$
  that belong to several different sets $X_{\om}$ also has zero
  Lebesgue measure.

\begin{claim}
 Whenever
    $n\geq m$ and $\om\in I^m$, we have 
     \begin{equation} \label{eqn:measureestimate}
        m_n(X_{\om}) \leq K^t\cdot \frac{\|D\phi_{\om}\|^t}{Z_{m}(t)},
    \end{equation}   
where $K$ is the bounded distortion constant.
\end{claim}
\begin{subproof}
Let $\tau\in I^n$ be an extension of a word $\om'\in I^m$. This means that there
exists $\g\in I^{m+1,n}$ such that $\tau = \omega'\g$. Then
\[ \|D\phi_{\tau}\| 
\leq \|D\phi_{\om'}\|\cdot \|D{\phi^{m+1,n}_\g}\| 
\leq K\cdot \|D\phi_{\tau}\|. 
\]
by the bounded distortion condition.
In particular,
      \[ Z_{n}(t) \geq \frac{1}{K^{t}}\cdot Z_m(t)\cdot
      \sum_{\g \in I^{(m+1,n)}}\|D{\phi^{m+1,n}_\g}\|^t        \] 
    and
     \begin{align*}
       m_n(X_{\om}) &= \frac{\sum_{\g\in I^{m+1,n}} 
          \|D\phi_{\omega\g}\|^t}{Z_n(t)} \leq
          \|D\phi_{\om}\|^t\cdot \frac{\sum_{\g\in
              I^{m+1,n}}\|D{\phi^{m+1,n}_{\g}}\|^t}{Z_n(t)} 
          \\ &\leq \|D\phi_{\om}\|^t\cdot
          \frac{K^t}{Z_m(t)}. \qedhere \end{align*} 
  \end{subproof} 
Now fix $r>0$ and let $B$ be a ball in $\R^d$ of radius
$r>0$. Consider the set $W$ of finite words $\om$ such that:
 \begin{itemize}
   \item $B\cap\(X_{\om}\)\neq\emptyset$;
   \item $\diam\(X_{\om}\)\geq r$;
   \item $\diam\(X_{\om a}\)\leq r$ for some  $a\in I^{(|w|+1)}$.
 \end{itemize}
  Let
   $W'$ consist of those words in $W$ that are not extensions of
   some other word in $W$. Then the sets $X_{\om}$, with $\om\in W'$,
   have pairwise disjoint interiors and  
   diameter at least equal to $r$. 
  By the uniform distortion property and the assumption on the geometry of $X$,
   there is a constant $M$, independent of $r$, such that no more than
   $M$ of the sets $X_{\om}$, $\om\in W$, can be pairwise disjoint. (The reason is that
   each set $X_{\om}$ contains a cone of definite opening angle, based at a point of
   $B$ and of diameter comparable to $r$; see \cite[Formula (2.10)]{mulondon}. 
   The possible number of such cones that are pairwise disjoint must~-- as we see by considering their total volume~-- necessarily be bounded
    independently of $r$. For details, compare \cite[Lemma 2.7]{mulondon}, 
    which is stated for autonomous systems, but applies equally in our setting.)

In particular,
   $\# W' \leq M$, and for every $k$, the number of words of length
   $k$ in $W$ is also bounded by $M$. 

In the following, we use the term
   "$\const$" to denote a positive constant
   that may depend on $\Phi$ and $t$, but not on the ball $B$ or the
   sets $W$ and $W'$. 

Fix $\tau\in W'$ and put $k=|\tau|$. For $m\geq 0$, we define
    \[ W_{\tau}(m) := \{ \omega \in W\cap I^{k+m}: \tau = \omega|_k\}. \]
 If $\omega\in W_{\tau}(m)$, then 
$$
\diam\(X_{\om|_{k+1}}\)
\leq \const\cdot r\cdot\rho_{k+1}
$$ 
and thus
$$
\diam(X_{\om})
\leq \const\cdot r\cdot \rho_{k+1}\cdot \eta^{m-1}
$$ 
by uniform contraction. Since $\diam(X_\om)>r$ by definition of $W$, we 
  see that $\rho_{k+1}\cdot \eta^{m}$ is bounded
below by a positive constant. Thus
    \[ m \leq \operatorname{const}\cdot (1+\log \rho_{k+1}). \]
Since all words in $W_{\tau}(m)$ have length $k+m$, we must have
  $\#W_\tau(m)\le M$, so we see that
\begin{align}\label{eqn:boundondepth} 
\# W_\tau& \leq \const \cdot (1+\log \rho_{|\tau|+1}), \quad\text{where } \\ \notag
W_\tau =&\bu_{m=0}^\infty W_\tau(m)=\bu
\{W_\tau(m) : 0\leq m \leq \const(1+\log\rho_{|\tau|+1})\}.
\end{align} 
In particular,
\begin{equation}\label{1na20120915}
 \Delta := \max_{\omega\in W} |\omega| < \infty.
\end{equation} 

Now let $\om\in W$ and set $n:= |\om|$. Let $A_{\om} \subset I^{(n+1)}$
be the set of all $a\in I^{(n+1)}$ for which 
$X_{\om a}$ intersects $B$ and $\diam(X_{\om a})\leq r$. Then
\begin{equation} \label{eqn:doubleunion}
B\cap J(\Phi)  
\subset \bigcup_{w\in W} \bigcup_{a\in A_{\om}}X_{\om a},
\end{equation}
and the sets in the union on the right hand side have pairwise
disjoint interiors. 
Since the disk of radius $2r$ and with the same center as $B$ contains
$\bigcup_{a\in A_{\om}} X_{\om a}$, we have 
\begin{equation}\label{eqn:restimate}
 r^d \geq \const\cdot  \sum_{i\in A_{\om}} \diam(X_{\om i})^d 
       \geq \const \cdot \|D\phi_\om\|^d\cdot \sum_{a\in
           A_{\om}} \|D\phi_a^{(n+1)}\|^d. 
\end{equation}
Let $q\geq \Delta+1$. 
  We are going to estimate the measure of the inner union in
   (\ref{eqn:doubleunion}) with respect to the measure $m_q$.
  First note that that using~\eqref{eqn:measureestimate} and
   the definition of $\Delta$, we get from
  \eqref{eqn:measureestimate} that
\begin{align*}
m_q\lt(\bigcup_{a\in A_{\om}} X_{\om a}\rt) 
     &\leq 
         \sum_{a \in A_{\om}} m_q(X_{\om a}) 
     \leq 
         \frac{K^t}{Z_{n+1}(t)}\sum_{a\in A_{\om}} \|D\phi_{\om a}\|^t
     \\ &\leq
         \frac{K^t\cdot \|D\phi_{\om}\|^t}{Z_{n+1}(t)}\sum_{a\in
           A_{\om}} \|D\phi_{a}^{(n+1)}\|^t.
  \end{align*} 
 By \eqref{eqn:restimate}, we can estimate
 \[ \|D\phi_\om\| \leq \const\cdot \frac{r}{\sqrt[d]{\sum_{a\in
       A_{\om}} \|D\phi_a^{(n+1)}\|^d}}, \] 
 so
\begin{equation}\label{eqn:lowerboundmoreprecise}
         m_q\lt(\bigcup_{a\in A_{\om}} X_{\om a}\rt) 
     \leq  \frac{\const\cdot r^t}{Z_{n+1}(t)}\cdot 
                      \frac{\sum_{a\in A_{\omega}}\|D\phi_{a}^{(n+1)}\|^t}%
                             {\left(\sum_{a\in
                             A_{\omega}}\|D\phi_{a}^{(n+1)}\|^d\right)^{t/d}}.  
\end{equation} 
 The fraction on the right can be estimated using H\"older's
 inequality: putting $\theta := d/t\geq 1$, we see that 
\[ 
\frac{\sum_{a\in A_{\om}} \|D\phi_a^{(n+1)}\|^t}
     {\left(\sum_{a \in A_{\om}}\|D\phi_a^{(n+1)}\|^d\right)^{t/d}} 
=\frac{\sum_{a\in A_{\om}}\|D\phi_a^{(n+1)}\|^t}
      {\left(\sum_{a \in
            A_{\om}}(\|D\phi_a^{(n+1)}\|^t)^{\theta}\right)^{1/\theta}}
 \leq  
         (\# A_{\om})^{1-1/\theta} \leq (\# I^{(n+1)})^{1-t/d}.
\]
So   
\begin{equation}\label{eqn:lowerboundgrowth}
         m_q\lt(\bigcup_{a\in A_{\om}} X_{\om a}\rt) 
     \leq  \frac{\const\cdot r^t}{Z_{n+1}(t)}\cdot (\# I^{(n+1)})^{1-t/d}. 
\end{equation}
Now $Z_{n+1}(t)\geq \const\cdot Z_n(t) \cdot \# I^{(n+1)}\cdot {\un
  c}_{n+1}^t$, and therefore  
\begin{align*}
         m_q\lt(\bigcup_{a\in A_{\om}} X_{\om a}\rt) 
     &\leq  \frac{\const\cdot r^t}{Z_{n+1}(t)}\cdot (\# I^{(n+1)})^{1-t/d}  \\
      &\leq 
            \frac{\const}{Z_n(t)\cdot (\# I^{(n+1)})^{t/d}\cdot {\un
                c}_{n+1}^t} = \frac{\const}{\tilde{Z}_{n+1}(t)}\cdot
            r^t. \end{align*} 
Now let $\tau\in W'$, and $\om\in W_{\tau}$. Then, by the assumption
of the theorem, we get 
\begin{align*}
m_q(\bigcup_{a\in A_{\om}} X_{\om a}) 
&\leq \frac{\const}{\Zt_{n+1}(t)} \cdot r^t 
\leq \frac{\const}{1+\log\max_{j\leq n+1} \rho_j}\cdot r^t \\
&\le\frac{\const}{1+\log\rho_{|\tau|+1}}\cdot r^t,
\end{align*}
provided that $r$ was chosen sufficiently small, and hence $n=|\om|$ is
sufficiently large. So we can apply (\ref{eqn:boundondepth}) to see that
$$
\aligned
m_q(B) 
&\le m_q\lt(\bu_{\tau\in W'}\bu_{\om\in W_\tau}\bu_{a\in A_{\om}}X_{\om
    a}\rt)
\le \sum_{\tau\in W'}\sum_{\om\in W_\tau}m_q\lt(\bu_{a\in A_{\om}}X_{\om a}\rt) \\
&\le \sum_{\tau\in W'}r^t\sum_{\om\in W_\tau}
       \frac{\const}{1+\log\rho_{|\tau|+1}} \\
&\le \const r^t\sum_{\tau\in W'}\frac{\#W_\tau}{1+\log\rho_{|\tau|+1}} \\
&\le \const r^t.
\endaligned
$$
Hence, if $m$ is an arbitrary weak limit of the sequence
$(m_q)_1^\infty$, then 
$$
m(B)\le \const r^t
$$
and, by the Converse Frostman Lemma, it follows that
$\hdim(J(\Phi))\geq t$, as claimed. 
\end{proof}
\begin{remark}
   It is tempting to replace the rather coarse bound in the definition
   of $\tilde Z_n(t)$, which uses the minimal size of a piece at level $n$,
   by  
     \[ \sqrt[d]{\sum_{a\in I^{(n)}} \|D\phi_a^{(n)}\|^d}. \]
Unfortunately, with this alternative definition,
Theorem~\ref{thm:lowerbound} no longer holds. We shall leave it to
the reader to construct a counterexample, which can be done along
similar lines as in Section \ref{sec:counterexamples}. 

  However, in view of \eqref{eqn:lowerboundmoreprecise} we \emph{can}
  replace $\tilde{Z}_n(t)$ in the 
   statement of the theorem by
     \[ \hat{Z}_n(t) := Z_n(t)\cdot 
            \min_{A\subset I^{(n)}} 
                   \frac{\sum_{a\in A}
                     \left(\|D\phi_a^{(n)}\|^d\right)^{t/d}}{\sum_{a\in
                       A} \|D\phi^{(n)}_a\|^t}. \] 
   We shall revisit this idea in the proof of Theorem \ref{thm:extremal}.
\end{remark}

\section{Balancing conditions}\label{sec:balance}

We now use Theorem \ref{thm:lowerbound} to establish Bowen's formula in certain cases. 
In particular, we shall prove Corollary~\ref{cor:weaklybalancedbowen},
an important special case of Theorem~\ref{thm:bowensubexponential} that
 is essential for the proof of the full theorem. 

We 
  need to make assumptions on
  how much the sizes of the different sets (as measured by the
  norm of derivatives) vary at a fixed level $n$; i.e., we need to bound the quantity 
  $\rho_n$ that was defined in \eqref{eqn:defnrho}:
 \[
  \rho_n := \sup_{a,b\in I^{(n)}}
        \frac{\|D{\phi^{(n)}_a}\|}{\|D{\phi^{(n)}_b}\|} \geq 1.
     \]
 We note, however, that, in the next section, we will be able to 
  remove 
  the hypotheses on $\rho_n$ from many of the results we prove here.

\begin{defn}[Balancing conditions]\label{defn:balance}
 A non-autonomous conformal iterated function
   system is called \emph{perfectly balanced} if $\rho_n=1$ for all $n$ and
   furthermore all $\phi_i^{(n)}$ are affine similarities.

  The system is called 
  \emph{balanced} if there is a constant $\kappa\geq 1$ such that  
   $\rho_n\leq \kappa$
   for all $n\in\N$, and \emph{weakly balanced} if 
    \[ \lim_{n\to\infty} \frac{1}{n}\log \rho_n = 0. \]

 Finally, we call $\Phi$ \emph{barely balanced} if 
    \[ \lim_{n\to\infty} \frac{\log(1+\log\rho_n)}{n}= 0.\]
\end{defn}

  For the purposes of the estimates on Hausdorff dimension in this paper,
   we shall see that weakly balanced systems do not differ from
   perfectly balanced ones. That is, all our positive results
   apply to weakly balanced systems, 
   while all weakly-balanced counterexamples
   can be constructed to be perfectly balanced.  

 The ``barely balanced'' condition is of importance mainly because, for
  these systems, 
  Theorem~\ref{thm:lowerbound} becomes a lower bound depending only on the
  behavior of the functions $\tilde Z_n(t)$.
  Recall that this quantity was defined in \eqref{eqn:Ztilde} as 
\[ 
\tilde Z_n(t) 
:= Z_{n-1}(t) \cdot (\# I^{(n)})^{\frac{t}{d}}\cdot 
               {\un c}_n^t
\]

\begin{lem}[Lower bound for barely balanced systems]\label{c120111111}
   \label{lem:barely}
  Let $\Phi$ be a finite NCIFS.
  \begin{enumerate}
    \item The function 
  \[  \tilde P(t):=\liminf_{n\to\infty}\frac1n\log\tilde Z_n(t) \]
      is strictly decreasing when it is finite; in fact, 
     \[  \tilde{P}(t') \geq \tilde{P}(t) + (t-t')\cdot \log\frac{1}{\eta} \]
  when $t'<t$ and $\tilde{P}(t)$ is finite. 
   (Here $\eta$ denotes the contraction constant, as usual.)
   \item
      Suppose that $\Phi$ is barely balanced,
     and that $t\ge 0$ is such that
      \[
      \tilde P(t):=\liminf_{n\to\infty}\frac1n\log\tilde Z_n(t)\ge 0. \]
       Then $\hdim(J(\Phi))\geq t$.
   \end{enumerate}
\end{lem}
\begin{proof}
  By definition of ${\un c}_n$ and Bounded Distortion, 
   the image of
  each $\phi_i^{(n)}$ contains a ball of radius comparable to ${\un c}_n$.
  By the Open Set Condition, it follows that 
  there exists a constant $C>0$ such that
 \[
     \#I^{(n+1)}\cdot {\un c}_{n+1}^d           \le C 
 \]
for all $n\ge 0$. Hence, for $t'<t$,
\[
      \frac{\tilde Z_{n+1}(t)}{\tilde Z_{n+1}(t')}
    =
      \frac{Z_n(t)}{Z_n(t')}\((\#I^{(n+1)})^{1/d}{\un c}_{n+1}\)^{t-t'}
    \le 
      \frac{Z_n(t)}{Z_n(t')}C^{\frac1d(t-t')}
    \le 
      C^{\frac1d(t-t')}\eta^{(t-t')n}
\]
for sufficiently large $n$. So
\begin{align*}
   \tilde{P}(t') = \liminf_{n\to\infty}\frac{1}{n+1} \tilde Z_{n+1}(t') &\geq 
   (t-t')\cdot \log(1/\eta) + 
      \liminf_{n\to\infty}\left(
            \frac{\tilde{Z}_n(t)}{n} - \frac{(t-t')\log C}{dn}\right)  \\
       &= (t-t')\cdot \log(1/\eta) + \tilde{P}(t). \end{align*}
      
 This proves the first claim. Now suppose that $\Phi$ is barely balanced
   and that $\tilde{P}(t)\geq 0$. Let $t'<t$ be arbitrary. Then 
   $\tilde{P}(t')>0$, so $\tilde{Z}_n(t')$ grows at least exponentially in $n$,
   while the denominator 
   in the hypothesis of Theorem~\ref{thm:lowerbound} grows subexponentially
   since $\Phi$ is barely balanced. So $\HD(J(\Phi))\geq t'$     
 by Theorem \ref{thm:lowerbound}. Since $t'<t$ was arbitrary,
 we have obtained the desired conclusion. 
\end{proof}

For the remainder of the section, we shall study weakly balanced systems.

\begin{cor}[Bowen's formula for weakly balanced
  systems]\label{cor:weaklybalancedbowen} 
 Let $\Phi$ be
  a nonautonomous iterated function system that is 
   subexponentially bounded and weakly balanced. 
   Then $\hdim(J(\Phi))=B(\Phi)$.
\end{cor}
\begin{proof}
 Let $t<B(\Phi)$. Then $\lP(t)>0$, and hence  
 $Z_n(t)$ grows at least exponentially in $n$. We have
 \begin{align*}
Z_{n+1}(t) \leq Z_n(t) \cdot \sum_{a\in I^{(n+1)}} \|D\phi_a^{(n+1)}\|^t &\leq
 Z_n(t)\cdot \# I^{(n+1)} \cdot (\rho_{n+1} \cdot \un c_{n+1})^t \\
   &=\tilde Z_{n+1}(t) \cdot (\#I^{n+1})^{1-t/d} \cdot \rho_{n+1}^t. 
\end{align*}
Since the system $\Phi$ is subexponentially bounded and weakly balanced,
$(\#I^{(n+1)})^{1-t/d}\rho_{n+1}^t$ grows at most subexponentially in
$n$. So
  \[ \liminf_{n\to\infty} \frac{1}{n} \log \tilde{Z}_{n}(t) \geq
     \liminf_{n\to\infty} \frac{1}{n} Z_{n}(t) = \lP(t)>0. \]
 Corollary~\ref{c120111111} implies that $\hdim(J(\Phi))\geq t$.

Since $t<B(\Phi)$ was arbitrary, we see that $\hdim(J(\Phi))\geq B(\Phi)$.
Together with Lemma~\ref{lem:HDupperBound}, this completes the proof.
\end{proof}

As we shall see in the next section, 
 the condition that $\Phi$ is subexponentially bounded
 cannot be relaxed:  there are 
 perfectly balanced 
  systems where Bowen's formula fails and $\# I^{(n)}$ has arbitrarily small
  exponential growth. 

 However, for these examples the contraction constants, along with the numbers
  $\# I^{(n)}$, behave very irregularly: for many values of $n$, $\ov{c}_n$ is bounded away from zero and 
  $\# I^{(n)}$ is finite. The next result, another interesting consequence of
  Theorem~\ref{thm:lowerbound}, establishes 
  part of Theorem \ref{thm:bowenexponential} for 
  (sufficiently balanced) systems of at most exponential growth when ${\ov c}_n\to 0$.

\begin{cor}[Systems of exponential growth]\label{cor:bowenexponential}
  Suppose that $\Phi$ is a nonautonomous conformal IFS such that
     \[
        \limsup_{n\to\infty} \frac{1}{n}\log \# I^{(n)} < \infty 
    \quad\text{and}\quad
        \limsup_{n\to\infty} \frac{1}{n}\log \rho_n  < \infty.
     \]

  If $\lim_{n\to\infty} {\ov c}_n = 0$, then Bowen's formula holds. (Recall that ${\ov c}_n = {\un c}_n \cdot \rho_n$ is the largest derivative
    of a map in $\Phi^{(n)}$.)
\end{cor}
\begin{proof}
  The key observation is that $Z_n(t)$ grows superexponentially for $t< B(\Phi)$, which is enough to show that $\tilde{P}(t)>0$. 

  Indeed, 
   fix $t< B(\Phi)$ 
 and let $\eps>0$ and $t':=t-\eps$. For large $n$, we have
     \[ Z_n(t') = \sum_{\omega\in I^n} \|D\phi_{\omega}^n\|^{t-\eps} 
      \geq
         Z_n(t) \cdot 
           \left(\max_{\omega\in I^n} \|D\phi_{\omega}^n\|\right)^{-\eps} 
      \geq
             Z_n(t)\cdot \prod_{j=1}^{n} {\ov c}_j^{-\eps}
      \geq \prod_{j=1}^{n} {\ov c}_j^{-\eps}. \] 
   (Here $K$ is the distortion constant.) Since ${\ov c}_n\to 0$, we have 
      \[ \frac{\sum_{j=1}^n -\log{\ov c}_n}{n} \to\infty, \]
   and hence $(\log Z_n(t-\eps))/n\to\infty$. 
   As in the preceding proof, we have 
     \[ {\tilde Z}_n(t') \geq 
           \frac{Z_n(t')}{(\# I^{(n)})^{1-t'/d}\cdot \rho_n^{t'}}. \]
   We just saw that  the numerator in this expression grows superexponentially, while 
     the denominator grows at most exponentially fast by assumption.
     By Lemma \ref{lem:barely}, it follows that
      $\HD(J(\Phi))\geq t-\eps$. 
      As $t< B(\Phi)$ and $\eps>0$ were arbitrary, we see that
      $\HD(J(\Phi))\geq B(\Phi)$, as desired.  
\end{proof}
\begin{remark}[Remark 1]
  The requirement that ${\ov c}_n \to 0$ is satisfied whenever
    $\Phi$ is weakly balanced and $\liminf (\log\# I^{(n)})/n > 0$. 
\end{remark}
\begin{remark}[Remark 2]
 The proof shows that the assumption ${\ov c}_n\to 0$ can be replaced by
     \[ \frac{\sum_{j=1}^{n} \-\log{\ov c}_n}{n} \to\infty. \]
\end{remark}

In a similar vein, we can now prove Proposition \ref{prop:bowenregular}, which 
   shows that we can explicitly calculate the Hausdorff dimension
  whenever both the number of pieces at each level and their sizes have 
  regular exponential growth. Using our notation from this
  section, we can restate the result
  as follows. 

\begin{cor}\label{thm:bowenexponentialgeneral}
  Suppose that $\Phi$ is weakly balanced, and that furthermore both limits
    \[ a := \lim_{n\to\infty}\frac1n\log\# I^{(n)} \]
and
\[
b:=\lim_{n\to\infty}\frac{1}{n}\log(1/\ov{c}_n)
\]
exist and are finite and positive. Then $\hdim(J(\Phi))=B(\Phi)=a/b$. 
\end{cor}
\begin{proof}
 The previous Corollary implies that Bowen's formula holds, hence it remains to show that
   $B(\Phi)=a/b$. 

Because $\Phi$ is weakly balanced, we have
\[ b=\lim_{n\to\infty}\frac1n\log (1/\un c_n). \]
For any $t$, we have
\[
Z_{n+1}(t)\le Z_n(t)\cdot \#I^{(n+1)}\cdot {\ov c}_{n+1}^t \]
 and 
\[
Z_{n+1}(t)\ge \const \cdot Z_n(t) \cdot \#I^{(n+1)} \cdot {\un c}_{n+1}^t \]

So
    \[ \const \cdot \# I^{(n+1)}\cdot {\un c}_{n+1}^t \leq \frac{Z_{n+1}(t)}{Z_n(t)} \leq \# I^{(n+1)}\cdot {\ov c}_{n+1}^t. \]
By assumption, if $t>a/b$, the right-hand side tends to zero as $n\to\infty$. Likewise, for $t<a/b$, the left hand side tends to $\infty$.
   This implies that $Z_n(t)$ tends to $0$ for $t>a/b$ and to $\infty$ for $t<a/b$, and hence $B(\Phi)=a/b$, as claimed. 
\end{proof}
\begin{remark}
The proof above gives a litttle bit more. Namely, suppose that $\Phi$ is a nonautonomous conformal IFS such that 
\[
0 < a_-
:= \liminf_{n\to\infty}\frac{\log\# I^{(n)}}{n}
\le \limsup_{n\to\infty}\frac{\log\# I^{(n)}}{n} =: a_+ < \infty \]
and
\[ 0 < b_- := \liminf_{n\to\infty} \frac{\log(1/{\ov c}_n)}{n} \leq
    \limsup_{n\to\infty} \frac{\log(1/{\un c}_n)}{n} =: b_+ < \infty, \]
then
$$
\frac{a_-}{b_+}\le \hdim(J(\Phi))\le \frac{a_+}{b_-}.
$$
\end{remark}

Together with Corollary \ref{cor:bowenexponential}, the following result
  provides a preliminary version of Theorem \ref{thm:bowenexponential} for 
  sufficiently balanced systems of at most exponential growth.

\begin{thm}[Systems with extremal Hausdorff and Bowen dimension]
    \label{thm:extremal}
 Let $\Phi$ be a finite non-autonomous conformal
  iterated function system such that
    \[ \limsup_{n\to\infty} \frac{\log \# I^{(n)}}{n} < \infty \quad\text{and}\quad    \limsup \log \rho_n/n < \infty. \]

 If $\HD(J(\Phi))=0$, then $B(\Phi)=0$, 
    and if  $B(\Phi)=d$, then $\HD(J(\Phi))=d$.
\end{thm}
\begin{proof}
For every $t\ge 0$, we again consider
\[ \tilde{ P}(t)=\liminf_{n\to\infty}\frac1n\log\tilde Z_n(t). \]
  Let us also set
\[   \un a=\liminf_{n\to\infty}\frac{\log\#I^{(n)}}{n} \geq 0,\quad\text{and}\quad
    \ov a=\limsup_{n\to\infty}\frac{\log\#I^{(n)}}{n} < \infty. \]

 For the first implication of the theorem, let us suppose that
   $B(\Phi)>0$, and that $\Phi$ is not necessarily weakly bounded, but that
\[  {\ov \rho} := \limsup_{n\to\infty} \frac{\log \rho_n}{n} < \infty. \]
  We must show that $\HD(J(\Phi))>0$. 
    To do fix $0<t_0<B(\Phi)$, so that $\lP(t_0)>0$. Recall that
    \[  
        Z_{n+1}(t_0) 
     \leq 
        Z_n(t_0) \cdot \# I^{(n+1)} \cdot {\ov c}_n^{t_0} 
     =  
        Z_n(t_0) \cdot \# I^{(n+1)}\cdot \rho_n^{t_0}\cdot {\un c}_n^{t_0}. 
    \]
   It follows that
     \[  
         \liminf_{n\to\infty} \frac{\log Z_n(t_0) + t_0\cdot \log {\un c}_n}{n} 
      \geq 
         \lP(t_0) - {\ov a} - t_0\cdot {\ov \rho} 
      > 
         -\infty. \]

  Now let $t\in (0,t_0)$, and observe that
    \begin{align*}
       \log {\tilde Z}_{n+1}(t) 
     &= 
       \log Z_n(t) + \frac{t}{d} \log \# I^{(n+1)} + t\cdot \log {\un c}_n  \\ 
     &> 
       \log Z_n(t_0) + \frac{t}{d} \log \# I^{(n+1)} + t\cdot \log {\un c}_n  \\ 
     &=
       \left(1-\frac{t}{t_0}\right)\cdot \log Z_n(t_0) 
       + 
       \frac{t}{d} \log \# I^{(n+1)} 
       + \frac{t}{t_0}\cdot (\log Z_n(t_0) + t_0\cdot \log {\un c}_n).
     \end{align*}
    Thus
    \[ 
         \tilde{P}(t) 
      \geq 
          \left(1-\frac{t}{t_0}\right) \cdot \lP(t_0) 
         + 
          \frac{t}{d}\cdot \underline{a} +  \frac{t}{t_0} \cdot 
                         (\lP(t_0) - {\ov a} - {\ov \rho}).
     \] 
   The right hands side tends to $\lP(t_0)$ 
    as $t\to 0$; in particular,
    $\tilde{P}(t)$ is positive  for sufficiently small positive $t$.
    Hence $\HD(J(\Phi)) \geq t > 0$, as claimed.

 \smallskip

 For the second implication, suppose that $B(\Phi)=d$. 
Let $\tilde{t}<d$, and choose $t\in(\tilde{t} , d)$ sufficiently close to $d$, as indicated below.
  We recall the proof of Theorem \ref{thm:lowerbound}. Let $m$ be the measure constructed there, let $B$ again
  be a ball of radius $r>0$, and continue the proof analogously up to 
  \eqref{eqn:lowerboundgrowth}:
\[          m\lt(\bigcup_{a\in A_{\om}} X_{\om a}\rt) 
     \leq  \frac{\const\cdot r^t}{Z_{n+1}(t)}\cdot (\# I^{(n+1)})^{1-t/d}. \]
  Recall that $\# I^{(n+1)}$ grows at most exponentially fast. Also recall that $n=|\om|$, where $\om$ was such that
   \[ r \leq \diam(X_{\om}) \leq \const\cdot \|D\phi_{\om}\| \leq \const\cdot \eta^n. \]
   Hence
\[          m\lt(\bigcup_{a\in A_{\om}} X_{\om a}\rt)  \leq
           \frac{\const\cdot r^{\tilde{t}}}{Z_{n+1}(t)} \cdot r^{t-\tilde{t}} \cdot (\# I^{(n+1)})^{1-t/d}
    \leq \frac{\const\cdot r^{\tilde{t}}}{Z_{n+1}(t)} \cdot \eta^{n(t-\tilde{t})} \cdot (\# I^{(n+1)})^{1-t/d}. \]
    If $t$ was chosen sufficiently close to $d$ (depending on $\eta$, the exponential growth rate of $I^{(n+1)}$ and $\tilde{t}$), the second term
    in the product tends to zero faster than the final term tends to infinity. Hence we have, for large $n$,
\[          m\lt(\bigcup_{a\in A_{\om}} X_{\om a}\rt)  \leq
  \frac{\const\cdot r^{\tilde{t}}}{Z_{n+1}(t)}. \] 
   
  Now let $\tau\in W'$ and $\omega \in W_{\tau}$. Since $\rho_n$ grows at most exponentially, and
   $Z_n(t)$ grows at least exponentially, we see that
    \[ m\lt(\bigcup_{a\in A_{\om}} X_{\om a}\rt)  \leq
  \frac{\const\cdot r^{\tilde{t}}}{Z_{n+1}(t)} \leq
           \frac{\const}{1+\log \rho_{|\tau|+1}}\cdot r^{\tilde{t}}. \]
   We now continue as in the proof of Theorem \ref{thm:lowerbound}, and see that 
     $\HD(J(\Phi))\geq \tilde{t}$. Since $\tilde{t}<d$ was arbitrary, we are done.
\end{proof}

\section{Approximation by subsystems}\label{sec:approximation}

 In this section, we shall study the approximation of a
    nonautonomous system $\Phi$
   (with finite or infinite alphabets) by suitably 
   chosen finite subsystems. The purpose of this procedure is two-fold. On the
   one hand, we will be able to remove the balancing 
   assumptions in many of the results of the preceding sections. Indeed,
   essentially we will show that we can always restrict to a subsystem
   that \emph{does} satisfy these assumptions, and whose pressure function
   is close to the original one. In particular, we complete the proof
   of   Theorem~\ref{thm:bowensubexponential} in this section.
On the other hand, the methods developed are crucial for our study
    of infinite systems in the second half of the 
   paper.


 The key result is as follows.
 \begin{prop}[Approximation by finite subsystems]
    \label{p1na20111107}\label{prop:approximation}
 Consider a (possibly infinite) non-autonomous iterated
 functions system $\Phi$. Let $t\in [0,d]$ and assume that 
 \begin{equation} \label{eqn:singlesums}
       \sum_{j\in I^{(n)}} \|D{\phi_j}\|^t < \infty
\end{equation}
   for all $n$. Let $\delta>0$, and, for each $n$, let
   $I^{(n)}_f\subset I^{(n)}$ be finite such that
     \begin{equation} \label{eqn:sumestimate}
         \sum_{j\in I^{(n)}}\|D\phi_j\|^t \leq 
           (1+\delta)\sum_{j\in I_f^{(n)}}\|D\phi_j\|^t 
     \end{equation}
   for all sufficiently large $n$.
  
  If $\Phi_f$ is the system obtained by using 
   $I^{(n)}_f$ as index sets instead of
     $I^{(n)}$, then 
    \[ \lP_f(t) := \lP^{\Phi_f}(t) \geq \lP(t) - \delta\cdot K^{2d}. \]
 \end{prop}
 \begin{proof}
  We may assume without loss of generality
   that (\ref{eqn:sumestimate}) holds for all $n$.
   If the system $\Phi$ 
   consists of affine similarities, then the result is
   immediate from the definitions. 
   In the  
   case where the maps are nonlinear, the idea is the same, but
   we need to 
   deal with the distortion constants. To do so, we must introduce
   notation that allows us to decompose a word $\omega$ into those
   symbols that belong to 
   the ``finite'' parts $I^{(n)}_f$ of the index sets, and those that do
   not. Let us fix $n\in\N$ in the following.

  For $n\geq 1$ and $0\leq s \leq n$, 
   let $\cP_s$ be the
   collection of all subsets of $\N_n = \{1,\dots,n\}$ having exactly
   $s$ elements. Given  
   $P\in \cP_s$, we denote by 
   $I^P$ the set of all ``words'' 
   $(\om_j)_{j\in P}$ with $\om_j\in I^{(j)}$ for all $j\in P$, and similarly
 \[ 
I^{P}_f
    := \{ (\om_j)_{j\in P}: \om_j\in I^{(j)}_f\}; \qquad
I^{P}_{\infty} 
    := \{ (\om_j)_{j\in P}: \om_j\in I^{(j)}\setminus I^{(j)}_f\}. \]

 Now fix $s\in \{0,\dots,n\}$ and $P\in \cP_s$. 
 We can combine two words $\omega \in I^P$ and $\tau \in
 I^{\N_n\sms P}$ to a conventional word of length $n$, 
 which we denote $\omega \star \tau \in I^n$. More precisely,
\[ (\om \star \tau)_j =
\begin{cases}
  \om_j  
       &\text{if } j\in P \\
  \tau_j
       &\text{if } j\notin P. .
\end{cases}
\] 

Let
 $k\leq n+1-s$ be the number of maximal continuous segments of integers in
 $P$. Then one can think of $\omega \in I^P$ as consisting of 
    $k$ separate words $\omega^1,\dots,\omega^k$, with
    $\omega^i\in I^{m_i,n_i}$ (i.e., $\omega^i$ starts at index $m_i$ and 
    ends at index $n_i$), where $m_i\leq n_i<m_{i+1}$. 
    We abbreviate
     \[ \|D\phi_{\omega}^{(P)}\| := 
           \prod_{i=1}^k
              \|D\phi_{\omega^i}^{m_i,n_i}\|. \]
For any $\tau\in I^{\N_n\setminus P}$, we have
 \begin{equation}\label{eqn:splitting}
\|D\phi_{\omega\star \tau}\| 
\leq \|D\phi_{\omega}^{(P)}\| \cdot \prod_{j\in \N_n\sms P}\|D\phi^{(j)}_{\tau_j}\| 
\leq K^{k+n-s-1} \cdot\|D\phi_{\omega\star \tau}\|
\leq K^{2(n-s)} \cdot\|D\phi_{\omega\star \tau}\|. 
\end{equation}
Note that here we split the word $\omega\star\tau$ into $k+n-s$ pieces, each of which is one of the words $\omega^i$ or one of the symbols of
   $\tau$. Hence, in the second inequality, we applied the bounded distortion property
   precisely $k+n-s-1$ times, resulting in the stated factor. 

With these preparations, we can now estimate,
for each $P$, the contribution $Z^{(P)}_n(t)$ to the sum $Z_n(t)$ of the words having exactly $s$ entries in $I^P_f$, in the positions prescribed by $P$.
  We begin by splitting each of these words as in~\eqref{eqn:splitting} and rearranging:
\begin{align*}
   Z^{(P)}_n(t) :=  \sum_{(\om,\tau)\in I^P_f\times I^{\N_n\sms P}_{\infty}}\|D\phi_{\om \star \tau}\|^t 
 &\leq 
     \sum_{\om\in I^P_f} \|D\phi_{\omega}^{(P)}\|^t
      \sum_{\tau\in I^{\N_n\sms P}_{\infty}} 
        \prod_{j \in \N_n\sms P} \|D\phi^{(j)}_{\tau_j}\|^t \\ 
&= \sum_{\om\in I^P_f} \|D\phi_{\omega}^{(P)}\|^t
       \prod_{j\in \N_n\sms P} \sum_{a \in I^{(j)}\setminus I^{(j)}_f}
       \|D\phi^{(j)}_{a}\|^t.
\end{align*}
 By assumption, we can estimate the final sum as follows:
  \begin{align*}
   Z^{(P)}_n(t) &\leq \sum_{\om\in I^P_f} \|D\phi_{\omega}^{(P)}\|^t
       \prod_{j\in \N_n\sms P} \delta \sum_{a \in I^{(j)}_f}
       \|D\phi^{(j)}_{a}\|^t    \\ 
  &= \delta^{n-s} \sum_{\om\in I^P_f} \|D\phi_{\omega}^{(P)}\|^t
       \prod_{j\in \N_n\sms P} \sum_{a \in I^{(j)}_f} \|D\phi^{(j)}_{a}\|^t . \end{align*}
  Finally, we apply the second inequality of~\eqref{eqn:splitting} to recombine the words:
\[ Z^{(P)}_n(t) \leq \delta^{(n-s)} K^{2t(n-s)} \sum_{\om \in I^{\N_n}_f}\|D\phi_{\om}\|^t 
   =\delta^{n-s} K^{2t(n-s)} Z_n^{\Phi_f}(t). \]

Thus we can bound the sum $Z_n(t)$, by summing over all possible sets $P$: 
  \begin{align*}
   Z_n(t) &= \sum_{|\om|=n} \|D\phi_{\om}\|^t =
    \sum_{s=0}^n \sum_{P\in\cP_s}
            \sum_{(\om,\tau)\in I^P_f\times I^{\N_n\sms P}_{\infty}}\|D\phi_{\om \star \tau}\|^t
          \\
   &\leq Z_n^{\Phi_f}(t) \cdot 
     \sum_{s=0}^n \sum_{P\in \cP_s}\delta^{n-s} K^{2t(n-s)}  
   = Z_n^{\Phi_f}(t) \cdot 
        \sum_{s=0}^n \binom{n}{s} (\delta K^{2t})^{n-s} \\
  &=(1 + \delta K^{2t})^n \cdot Z_n^{\Phi_f}(t) 
  \leq (1 + \delta K^{2d})^n\cdot Z_n^{\Phi_f}(t). 
\end{align*} 

 So 
 $\lP(t) \leq \lP_f(t) + \log(1+\delta K^{2d}) \leq \lP_f(t) + \delta K^{2d}$
 as claimed. 
\end{proof}

 We can use the preceding result to determine the Hausdorff dimension of
  $J(\Phi)$ whenever, for every $\delta>0$, the subsystem
  $\Phi_f$ can be chosen 
   such that Bowen's formula holds for $\Phi_f$. In particular, the result
   shows that very small pieces are irrelevent for the computation of
   the pressure function. This means that, given a bound on the number of
   pieces at each level, we can always restrict to a system
   satisfying a corresponding balancing condition:

\begin{cor}[Balance from growth restrictions] \label{cor:balance}
  Let $\Phi$ be a finite non-autonomous conformal IFS, and let $t_0>0$. Also
   let $\alpha_n\geq 1$ be any sequence such that $\alpha_n\to\infty$.

  Then 
   there exists a subsystem $\Phi_f$, with index sets
   $I_f^{(n)}\subset I^{(n)}$ of $\Phi$ such that 
   $\lP^{\Phi_f}(t)=\lP^{\Phi}(t)$ for all $t\geq t_0$, and such that 
    \[ \rho_n(\Phi_f) \leq \alpha_n \cdot (\# I^{(n)})^{\frac{1}{t_0}} \]
   for all $n$. Here
      \[ \rho_n(\Phi_f) := 
          \max_{a,b\in I^{(n)}_f} \frac{\|D\phi_a^{(n)}\|}{\|D\phi_b^{(n)}\|}. \] 
\end{cor}
\begin{proof}
 Let us set 
     \[
    \eps_n := \overline{c}_n\cdot \alpha_n^{-1}\cdot (\# I^{(n)})^{-\frac{1}{t_0}}
      \] and let $I_f^{(n)}$ consist of all indices
  $a\in I^{(n)}$ for which $\|D\phi_a^{(n)}\|\geq \eps_n$. Then
   $\rho_n(\Phi_f)$ has the required property. 

  Furthermore, for $t\geq t_0$, 
    \begin{align*} \sum_{j\in I^{(n)}\setminus I^{(n)}_f} \|D\phi_j^{(n)}\|^t &\leq
       \# I^{(n)}\cdot \eps_n^t \\ &=
       \overline{c}_n^t \cdot (\# I^{(n)})^{1-\frac{t}{t_0}} \cdot \alpha_n^{-t}
       \leq \alpha_n^{-t}\cdot \sum_{j\in I^{(n)}_f} \|D\phi_j^{(n)}\|^t. \end{align*}
   Since $\alpha_n$ tends to zero, it follows from
     Proposition \ref{prop:approximation} that
    $\lP^{\Phi_f}(t)=\lP^{\Phi}(t)$, as claimed. 
\end{proof}

This allows us to complete the proof of the positive part of Theorem
  \ref{thm:bowensubexponential}. 

\begin{thm}[Bowen's formula for systems of sub-exponential growth]  
\label{thm:bowensubexponential_positive}
Suppose that $\Phi$ is a  
 conformal iterated function system that is
  subexponentially bounded, i.e.
\[ 
\lim_{n\to\infty} \frac1n\log \# I^{(n)}= 0.  
\] 
Then Bowen's Formula holds for $\Phi$, i.e.
\[ \HD(J(\Phi))=B(\Phi). \]
\end{thm}
\begin{proof}
  Recall that always $\HD(J(\Phi))\leq B(\Phi)$ by 
   Lemma~\ref{lem:HDupperBound}. 
   If $B(\Phi)=0$, there is nothing to prove, so let us assume that 
    $B(\Phi)>0$, and fix some $t_0\in (0,B(\Phi))$.

   By Corollary \ref{cor:balance}, there is a subsystem $\Phi_f$ of $\Phi$
    such that $\lP^{\Phi_f}(t)=\lP^{\Phi}(t)$ for all $t\geq t_0$, and such that
    \[ \rho_n(\Phi_f) \leq n \cdot (\# I^{(n)})^{\frac{1}{t_0}}. \]
   In particular, $B(\Phi_f)=B(\Phi)$ and $\Phi_f$ is weakly balanced. 
   By Corollary \ref{cor:weaklybalancedbowen}, we see that
      \[ B(\Phi)=B(\Phi_f)=\HD(J(\Phi_f)) \leq \HD(J(\Phi)) \leq B(\Phi).
               \qedhere \]
\end{proof}

Likewise, we can complete the positive part of Theorem
  \ref{thm:bowenexponential}. 

\begin{thm}
  Suppose that $\Phi$ is a non-autonomous conformal iterated function system
   such that
    \[ \limsup_{n\to\infty} \frac{1}{n}\log \# I^{(n)} < \infty. \]
  If $\overline{c}_n\to 0$ as $n\to \infty$, if $\HD(J(\Phi))=0$, or if
   $B(\Phi)=2$, then
   Bowen's formula holds. 
\end{thm}
\begin{proof}
  As in the proof of the preceding theorem, there is
   a finite subsystem $\Phi_f$ where $\rho_n(\Phi_f)$ 
   grows at most exponentially, and
   for which $B(\Phi_f)=\Phi_f$. 
   The result follows from Corollary \ref{cor:bowenexponential} resp.\ 
    Theorem \ref{thm:extremal}.    
\end{proof}

\section{Counterexamples to Bowen's Formula} \label{sec:counterexamples}

We now turn to completing the proof of 
  Theorem \ref{thm:bowensubexponential} by giving counterexamples
  that show that Theorem
\ref{thm:bowensubexponential_positive} is best possible. 

 \begin{thm}[Counterexamples to Bowen's formula]\label{thm:counterexamples}
 Let $0 < t_1 < t_2 < d$, and let $\eps>0$. 
  Then there exists a perfectly balanced NCIFS 
  $\Phi$ with $\limsup_{j\to\infty}\frac1j\log\#I^{(j)}\leq \eps$ 
 such that 
$$
\hdim(J(\Phi))=t_1 \  \ \text{{\rm and }} \ B(\Phi)=t_2.
$$
\end{thm}  
\begin{proof}
The idea of the proof is very simple. We begin with a uniformly finite
and balanced system $\Psi$ for which $\lP^{\Psi}(t_2)=\uP^{\Psi}(t_2)=0$,
hence $\hdim(J(\Psi))=t_2$, and $\lP(t_1)=\uP(t_1) <\infty$. We then
modify the system $\Psi$ at a sequence $(n_k)_1^\infty$ of times,
where $n_k$ will grow sufficiently quickly, in such a way that the
modified system $\Phi$ still has 
$B(\Phi)=t_2$, but such that the corresponding sets $X_{n_k}^{\Phi}$ admit 
much more efficient covers in terms of $t_1$-dimensional Hausdorff measure. 

We now provide the details, assuming for simplicity that $d=1$; the
general case is 
completely analogous. Let $X=[0,1]$. The construction begins
 by choosing a suitable perfectly balanced system $\Psi$ that is periodic, i.e. 
$$
\Psi^{(j+m)}=\Psi^{(j)}
$$ 
for all $j\in\N$ and some $m\in\N$. In other words, up to a renormalization
  of time, $\Psi$ is a classical iterated function system. 
  This system is going to be
  chosen such that 
  $\lP^{\Psi}(t_2)=\uP^{\Psi}(t_2)=0$. In order to achieve arbitrarily
  slow exponential growth, $\Psi$ will be
  chosen so that furthermore 
  the number of words of length $n$ has
  small exponential growth in $n$, which can be 
  achieved by artificially inserting long stages where a single
  contraction by a factor close to one is applied at each step in
  time. To be explicit, let us define such a system $\Psi$ by
\[
  \Psi^{\ell} := \begin{cases}
      \{x\mapsto x/2 ; x\mapsto (x+1)/2 \} &\text{if }\ell\MOD m = 0; \\
      \{x\mapsto \lambda x\}  &\text{otherwise}.\end{cases} \]
Here the integer $m\ge 1$ will be chosen sufficiently large, depending
on $t_1$ and $t_2$ as indicated below, and 
     \[\lambda := 2^{-\frac{1-t_2}{t_2(m-1)}}. \]
Then, for all $n,k\geq 1$ and all $t\geq 0$, we have 
\begin{equation} \label{eqn:boundonsums}
       Z_n^{\Psi}(t) \leq Z_n^{\Psi}(0) =
       2^{\left\lfloor\frac{n}{m}\right\rfloor} \leq 2^{\frac{n}{m}} 
\end{equation}
 and 
\begin{equation} \label{eqn:sumsfort2} 
Z_{k\cdot m}^{\Psi}(t_2) =
     \left( 2\cdot \left(\frac{\lambda^{m-1}}{2}\right)^{t_2}\right)^k
     = \left( \lambda^{t_2\cdot(m-1)} \cdot 2^{1-t_2} \right)^k 
= 
                                                 \left( 2^{t_2-1}\cdot
                                                   2^{1-t_2}\right)^k
                                                 = 1. 
\end{equation} 

\begin{claim}[Claim 1]
We have $\lP^{\Psi}(t_2)=\uP^{\Psi}(t_2)=0$. In particular,
$B(\Psi)=t_2$, and $\Z^{\Psi}_n(t_1)$ grows exponentially. 
\end{claim}
\begin{subproof}
  It follows from \eqref{eqn:sumsfort2} and the very definition of
  $\Psi$ that 
     \[ \lambda^{t_2\cdot(m-1)} \leq Z_{n}^{\Psi}(t_2) \leq 1 \]
  for all $n\in\N$. Hence
     \[ \lim_{n\to\infty} \frac1n\log Z_n^{\Psi}(t_2) = 0, \]
   as claimed.
\end{subproof}

 For all $k\geq 1$, we now inductively define
  positive integers $n_k$ and $M_k$, and a real number
  $\lambda_k\in (0,1)$. These sequences 
   give rise a non-autonomous iterated function systems $\Phi$ as follows. If $j\neq n_k$ for all $k$,
   we set
     \[ \Phi^{(j)} := \Psi^{(j)}, \]
    while
     \[
       \Phi^{(n_k)} := \{x\mapsto \lambda_k x/M_k, x\mapsto \lambda_k(x+1)/M_k, \dots, x\mapsto \lambda_k(x+M_k-1)/M_k\}. \]
  The system $\Phi$ is perfectly balanced; let us use
   $c_n$ to denote the contraction factor ${\un c}_n={\ov c}_n$ at stage
   $n$. Note that $c_{n}=\lambda_k/M_k$ when $n=n_k$ and 
    that $c_n\in \{\lambda,1/2\}$ for all other
    values of $n$.  

  We now specify the inductive construction. The initial value $n_1>m$
  is arbitrary. Let $k\geq 1$ and suppose that
     $n_j$ has been defined for $1\leq j\leq k$, while  
     $\lambda_j$ and $M_j$ have been defined for $1\leq j<k$. Note that
     this determines $Z_{n_k-1}^{\Phi}(t)$ 
    for all $t$. Then we define
     \[ \lambda_k :=
     \frac{1}{\bigl(Z_{n_k-1}^{\Phi}(t_1)\bigr)^{1/t_1}}\quad\text{and}\quad 
            M_k := \left\lfloor\lambda_k^{-\frac{t_2}{1-t_2}}\right\rfloor. \]
   Finally, we choose $n_{k+1}>n_k$ sufficiently large that 
        \[ \log Z_{n_{k+1}-1}^{\Psi}(t_1)/2 < \log
        Z_{n_{k+1}-1}^{\Phi}(t_1)  < 2\log
        Z_{n_{k+1}-1}^{\Psi}(t_1). \] 
    In particular, $\lambda_k\to 0$ and $M_k\to \infty$. 
    We may also assume that the sequence $n_k$ is chosen so that
    $n_k/k \to \infty$.  

 \begin{claim}[Claim 2]
  We have $\lP^{\Phi}(t_2)=\uP^{\Phi}(t_2)=0$ and $\hdim(J(\Phi))=t_1$.
\end{claim}
\begin{subproof}
We first turn to considering the sum $Z_{n}^{\Phi}(t_2)$. Since the
system is linear, this sum is just the product 
$$
\prod_{j=1}^n c_j^{t_2}\#\Phi^{(j)}.
$$ 
When $j\neq n_k$ for any $k$, the $j$th factor above is the same as
the corresponding contribution in $Z_n^{\Psi}(t_2)$. For $j=n_k$,
     \[ c_{n_k}^{t_2}\cdot \#\Phi^{(n_k)} =
        \frac{\lambda_k^{t_2}}{M_k^{t_2}}\cdot M_k =
        \lambda_k^{t_2}\cdot \left\lfloor
          \lambda_k^{-\frac{t_2}{1-t_2}}\right\rfloor^{1-t_2} 
                         \in  [1-\lambda_k, 1].
\]
It follows that 
     \[Z_n^{\Psi}(t_2)/C^k \leq  Z_n^{\Phi}(t_2) \leq C^{k}\cdot Z_n^{\Psi}(t_2) \]
    for a suitable constant $C>1$, where $k=k(n)$ is the largest
    integer with $n_k\leq n$. So  
      \[ \frac1n|\log Z_n^{\Psi}(t_2) - \log Z_n^{\Phi}(t_2)| \leq
      \log C\cdot \frac{k(n)}{n} \to 0, 
\] 
     and hence $\lP^{\Phi}(t_2) = \uP^{\Phi}(t_2)= 0$ by Claim 1.

To show that $\HD(J(\Phi))\leq t_1$, 
 note that $\U_{n_k} := \{[0,\lambda_k]\}$ covers
 the images of all maps in $\Phi^{(n_k)}$. Furthermore, we have
   \[ Z_{n_k-1}^{\Phi}(t_1)\cdot \lambda_k^{t_1} =1 \]
  by definition. Hence 
  Lemma~\ref{lem:HDuppergeneral} implies that $\HD(J(\Phi))\leq t_1$.

Finally, we consider the modified sums $\tilde Z_n(t_1)$ as defined 
 in \eqref{eqn:Ztilde}. If $n\neq n_k$, then $\tilde Z_n(t_1)$ agrees with
$Z_{n-1}(t_1)$ up to a bounded factor, and $Z_{n-1}(t_1)\to\infty$ as
  $n\to \infty$. (Recall that $B(\Phi)=t_2>t_1$.) 
  For $n=n_k$, we have
 \[ \tilde Z_{n_k}^{\Phi}(t_1)  = Z_{n_k-1}^{\Phi}(t_1)\cdot
     M_k^{t_1}\cdot c_{n_k}^{t_1} = 
                                              Z_{n_k-1}^{\Phi}(t_1)\cdot
                                              \lambda_k^{t_1} = 1. 
\] 
 Hence $\tilde{P}(t_1)=\liminf \log\tilde{Z}_{n}(t_1)/n\geq 0$, and 
  $\HD(J(\Phi))\geq t_1$ by Lemma \ref{lem:barely}.
\end{subproof}
It remains only to estimate the growth of the index sets, which means that we
should estimate the growth of $M_k$ compared to $n_k$. We have 
    \begin{align} \label{eqn:growthestimate}
      \log M_k &\leq 
       \frac{t_2}{1-t_2}|\log \lambda_k| = 
       \frac{t_2}{t_1(1-t_2)} |\log Z_{n_k-1}^{\Phi}(t_1)| \\\notag
     &<
       \log 2 \frac{t_2}{t_1(1-t_2)} |\log Z_{n_k-1}^{\Psi}(t_1)| \\\notag
    &\leq 
       2\log 2 \frac{t_2}{t_1(1-t_2)}\frac{n_k}{m} \leq
       4 \cdot \frac{t_2}{m t_1(1-t_2)}n_k. \end{align}
    If $m$ was chosen sufficiently large, we thus have
     $\log M_k \leq \eps\cdot n_k$, as desired. 
 \end{proof}
\begin{remark}
  The set $J(\Phi)$ in the above proof is a ``partial homogeneous Moran set'', 
     in the sense
     defined in \cite{moransetsclasses}, and as such it is well-known how to 
     calculate its Hausdorff dimension; see 
     e.g.\ \cite[Theorem~C]{moransetsclasses}.  
     We decided to include a self-contained proof for the reader's convenience,
     particularly since the notation used in the above reference is 
     different from ours.
\end{remark}

With minor modification, the same construction can also be used to
 show that our results on systems with at most exponential growth
 cannot be extended to any larger growth conditions, 
 completing the proof of Theorem \ref{thm:bowenexponential}. 
\begin{thm}[Counterexamples of superexponential growth]
 Let $(\alpha_n)_{n\in\N}$ be any sequence of positive integers such that
  $\log \alpha_n/n\to \infty$. Then
  there exists a perfectly balanced NCIFS $\Phi$ such that:
  \begin{enumerate}
    \item $\# I^{(n)}\leq \alpha_n$ for all $n$;
    \item ${\ov c}_n\to 0$ as $n\to\infty$;
    \item $\HD(J(\Phi))=0$ and $B(\Phi)=d$.
  \end{enumerate}
\end{thm}
\begin{proof}
  For $\eps\in(0,d/2)$, 
   set $t_1 = t_1(\eps) := \eps$ and $t_2 := t_2(\eps) := d-\eps$. Also let
   $\Psi=\Psi(\eps)$ be a perfectly balanced autonomous linear system with
   $B(\Psi(\eps))=t_2(\eps)$
   such that the contraction factor of the linear maps in $\Psi$ tends to zero as
   $\eps\to 0$. For example, in dimension one we can take
    \[ \Psi = \{ x\mapsto \rho x , x\mapsto \rho(x+1) , x\mapsto
          \rho(x+L-1) \}, \qquad \rho = L^{-\frac{1}{t_2}}, \qquad L=L(\eps)\to \infty. \]
   
 We can apply the construction from the previous proof
   to $t_1$, $t_2$ and $\Psi$. The result is a nonautonomous system 
    $\Theta(\eps)$ whose Bowen dimension is $t_2$, whose
    limit set has Hausdorff dimension $t_1$, and such that
    the growth of $\Theta(\eps)$ is at most exponential. (According 
    to~\eqref{eqn:growthestimate}, the exponent of growth tends to infinity
    at most like  $\log L/\eps^2$.)
    
  Now let $\eps_k\to 0$. The desired system
   $\Phi$ is obtained by letting $\Phi^{(j)}$ agree with
     $\Theta(\eps_k)^{(j)}$ for $N_{k-1}\leq j < N_k$, where $N_k$ is
   an increasing sequence of natural numbers.
   Since $Z_{n}^{\Theta(\eps_k)}(t)\to \infty$ for all $t<t_2(\eps_k)$, we can
   ensure that $Z_n^{\Phi}(t)\to\infty$ for all $t<d$, and hence $B(\Phi)=d$. Similarly,
   using the same covers as in the proof of Theorem \ref{thm:counterexamples}, we ensure that
   $\HD(J(\Phi)=0$. By choice of $\Psi$, and by construction, ${\ov c}_n\to 0$ as $n\to\infty$.
   
  Finally (again, assuming $N_k$ is chosen to grow sufficiently quickly), we can ensure that
    $\# I^{(n)} \leq \alpha_n$ for all sufficiently large $n$ (since each $\Theta(\eps_k)$ has
    some finite exponential growth rate, but $\alpha_n$ tends to infinity superexponentially).
    By modifying $\Phi$ at finitely many initial stages, we can ensure that the condition is
    satisfied for \emph{all} $n$, as claimed. 
\end{proof}

 To conclude this section, we shall discuss a very natural system
  (of super-exponential growth) where Bowen's formula fails. 

 This example arise from the study of 
   \emph{continued fraction expansions} of real numbers.
  These expansions can be studied by considering 
    the system of conformal maps 
      \begin{equation}\label{eqn:continuedfraction}
       \phi_n:x\mapsto 1/(n+x) \quad (n\geq 1). 
     \end{equation}
   In order to fit into our framework, where we require uniform contraction,
    we should only allow $n\geq 2$, As we will
    study numbers
    whose continued fraction expansions tend to $\infty$, this is 
    not a serious restriction.

Consider, for some $K\geq 2$ and $\alpha>1$, the index sets
   \begin{equation} \label{eqn:continuedfractionindex}
      I^{(n)} := \{ j\in\N : K^{\alpha^n} \leq j < K^{\alpha^{n+1}}\}. 
   \end{equation}
 The limit set of the associated nonautonomous IFS then consists of those 
   numbers
   $x\in(0,1)$ for which the continued fraction expansion $(a_n)_{n\in\N}$ 
   satisfies $K^{\alpha^n}\leq a_n \leq K^{\alpha^{n+1}}$. 

 \begin{prop}
  For every $\alpha>1$ and $K\geq 2$, the nonautonomous conformal iterated function system    $\Psi$ defined by
     \eqref{eqn:continuedfractionindex} and \eqref{eqn:continuedfraction}
     satisfies 
      \[ \HD(J(\Psi)) = \frac{1}{1+\alpha} < B(\Psi) = \frac{1}{2}. \]
 \end{prop}
\begin{sketch}
   (In the following, we shall only indicate the order of magnitude of the quantities that occur. 
     Turning this sketch into a proof with rigorous estimates is a straightforward, but not very enlightening, exercise.)

   We begin by calculating the Bowen dimension. 
     Since  $\|\phi_j'\| =1/j^2$, this means that we should estimate the sum
     \[ \sum_{j=j_1}^{j_2} \|\phi_j'\|^t = \sum_{j=j_1}^{j_2} j^{-2t}, \]
     where $j_1 = \left \lceil K^{\alpha^n}\right \rceil$ and 
     $j_2 = \left\lceil K^{\alpha^{n+1}}\right \rceil-1$.
    For large $n$, the value of this sum can be approximated (up to an error that is
     negligible for our purposes) by the integral 
    \[ \int_{j_1}^{j_2} x^{-2t}dx = \frac{1}{1-2t}(j_2^{1-2t} - j_1^{1-2t}) = \frac{j_2^{1-2t}}{1-2t}\left(1 - \left(\frac{j_1}{j_2}\right)^{1-2t}\right). \]

  So the $j$-th summand in the sum 
     $Z_n(t)$ is on the order of magnitude of
     $K^{\alpha^{j+1}(1-2t)}$. Summing the geometric series in the exponent, $Z_n(t)$ itself
    behaves like
    \[ K^{(1-2t)\cdot \frac{\alpha^{n+2}}{\alpha-1}}. \]
  It follows that the pressure is positive infinite for $t<1/2$ and negative infinite for $t>1/2$.
   Hence the Bowen dimension $B(\Psi)$ is equal to $1/2$, as claimed.

   Luczak \cite{luczak} showed that the set of points whose continued fraction 
    expansion
    satisfies $a_n \geq K^{\alpha^n}$ has Hausdorff dimension $1/(1+\alpha)$, so 
     $\HD(J(\Psi)) \leq 1/(1+\alpha)$. 

  Thus it remains to prove that $\HD(J(\Psi)) \geq 1/(1+\alpha)$. We shall use Theorem \ref{thm:lowerbound}. First observe that
       \[ \# I^{(n)}\cdot {\un c}_n \geq \const\cdot K^{-\alpha^{n+1}}, \]
      which means that 
              \[ \log \tilde{Z}_n(t) = \log Z_{n-1}(t) + t\log(\# I^{(n)}\cdot {\un c})\]
       grows like
        \begin{align*}
    \left( \frac{(1-2t)\cdot \alpha^{n+1}}{\alpha-1} - t\alpha^{n+1}\right)\cdot \log K &=
            \frac{\alpha^{n+1}\log K}{\alpha-1}\cdot (1-2t - t\alpha + t) \\
   &= \frac{\alpha^{n+1}\log K}{\alpha-1}\cdot 
                              (1 - t(1+\alpha)). \end{align*}
       For $t<1/(1+\alpha)$, this quantity tends to $+\infty$ exponentially fast, so that $\tilde{Z}_n(t)$ grows superexponentially. 
   On the other hand, 
      \[ \rho_n \leq K^{2\alpha^{n+1}},  \]
      and hence the quantity
        \[ 1 + \log \max_{j\leq n} \rho_j \]
      grows only exponentially with $n$. Hence $\HD(J(\Psi)) \geq 1/(1+\alpha)$ by
       Theorem \ref{thm:lowerbound}.
\end{sketch}

\section{Hausdorff and Packing Measures}\label{hpm}

In this short section, we consider balanced and uniformly finite systems,
  where we can fully charcterize the cases when the
  $h$-dimensional Hausdorff measure $\H_h(J(\Phi))$ is finite, positive
  or infinite, where $h=B(\Phi)=\HD(J(\Phi))$. 
  The answer depends on the value of  $\liminf_{n\to\infty} Z_n(h)$.

\bthm\label{t1na17}
Let $\Phi$ be a balanced and uniformly finite
 non-autonomous conformal iterated function
 system, and let $h=B(\Phi)$. 

Then the $h$-dimensional Hausdorff measure $\H_h(J(\Phi))$ is infinite,
  finite or equal to zero according to whether 
   \[ \liminf_{n\to\infty} Z_n(h) \]
  is infinite, finite or equal to zero, respectively. 
\ethm 
\begin{proof}
The proof of Lemma~\ref{lem:HDuppergeneral} produces a constant $C>0$
such that that
$$
\H_h(J(\Phi)) \leq C\liminf_{n\to\infty} Z_n(h).
$$
Hence we only need to show that $\H_h(J(\Phi))$ is positive when 
 the $\liminf$ above is positive, and infinite when the $\liminf$ is infinite.
 To do so, we note that the proof of 
 Theorem \ref{thm:lowerbound} actually gives that 
 \[ 
m(B)\leq \const\cdot r^h 
\]
for any ball of radius $r$ when the assumptions of that
 theorem are satisfied, i.e. when
    \begin{equation}\label{eqn:lowerboundassumption}
        \liminf_{n\to\infty}\frac{\tilde{Z}_n(h)}{1+\log \max_{j\leq n}\rho_n} 
             \end{equation}
 is positive. 
 Furthermore, if this $\liminf$ is infinite,
 then the proof shows even that
   \[ \frac{m(B)}{r^h} \to 0 \]
  as $r\to 0$. By the Converse Frostman Lemma, this means that
  $\H_h(J(\Phi))$ is positive resp.\ infinite. 

 By assumption, our system is balanced, so $\rho_n$ is uniformly bounded.
  Furthermore, the system is uniformly finite, and hence
  $\tilde{Z}_n(h)$ differs from $Z_n(h)$ by at most a multiplicative 
  constant. This means that the $\liminf$ in \eqref{eqn:lowerboundassumption}
  is comparable to $\liminf_{n\to\infty} Z_n(h)$, and the theorem is proved. 
\end{proof}

\section{Systems with a Countably Infinite Alphabet}\label{scia}
\label{InfiniteBowen}

 We now turn to the study of non-autonomous systems with infinite
  countable 
  alphabets. For simplicity, we shall assume that all stages are infinite,
  and share the same alphabet. That is, $I^{(n)}=\N$ for all $n\in\N$. 
  The reader may find it helpful to imagine the pieces at level $n$ being 
  labelled in decreasing order of size (i.e.,\ $(\|D\phi_j^{(n)}\|)_{j\in\N}$ 
  is a decreasing sequence), although we will not formally require this.

 One of the technical differences that arise in the infinite 
  case is that some of the sums $Z_n(t)$ may be infinite (for finite $n$),
  This 
  will cause the sums $Z_{n'}(t)$ to be infinite for all $n'\geq n$, and
  hence makes them rather unsuitable for predicting the geometry of the
  limit set at small scales without prior restrictions.
Clearly this happens if and only if one of the individual sums
$$
Z_1^{(m)}(t) = \sum_{j\in I^{(m)}} \|D\phi_j^{(m)}\|^t
$$ 
is infinite for
  $m\leq n$. In practical applications, for fixed $t$ these sums
  will usually be either finite for all $m$ or infinite for all $m$,
  and in the following we shall consider only systems where this is the case.
  As soon as we deal with this technical issue,
  Corollary~\ref{BowenM} and 
  with Corollary~\ref{EVclass} provide us immediately
  with a large class of systems where Bowen's formula holds. 

\

\bdfn\label{d120120707}
We say that an infinite non-autonomus conformal iterated function system
$\Phi=\{\Phi^{(n)}\}_{n=1}^\infty$ belongs to class $\cM$ if the
following conditions are satisfied for all $t\in (0,d)$ and all $\eps>0$.
\begin{itemize}
\item[(a)] The sums $Z_1^{(n)}(t)$ are either infinite for all $n$ 
  or finite for all $n$.
\item[(b)] 
If $Z_1(t)<\infty$, then
  \[ \lim_{n\to\infty} \frac{\sum_{k\leq e^{\eps n}} \|D\phi_k^{(n)}\|^t}
                         {Z^{(n)}_1(t)} = 1. \]
\item[(c)] 
If $Z_1(t)=\infty$, then
  \[ \lim_{n\to\infty} \sum_{k\leq e^{\eps n}} \|D\phi_k^{(n)}\|^t = \infty. \]
\end{itemize}
\edfn

Of course the key point of this technical definition
  is that, given $\Phi\in \cM$, we can restrict to
  a well-controlled finite subsystem.

\begin{obs}
  If $\Phi\in\cM$, then there is a finite and subexponentially bounded 
   subsystem $\Phi_F$, determined by index sets $I_F^{(n)}\subset\N$, 
   such that $\lP^{\Phi_F}(t) = \lP^{\Phi}(t)$ for all $t>0$.    
\end{obs}
\begin{proof}
  First observe that, for each $n$, we may reorder the sequence
    $(\phi_k^{(n)})_{k\in\N}$ without affecting the limit set of $\Phi$, the
    pressure function, or the value $Z_1^{(n)}(t)$ (for all $t$). Among all possible such
    reorderings, the sum
    $\sum_{k\leq e^{\eps n}} \|D\phi_k\|^t$ will be minimized if 
   the sequence $(\|D\phi_k^{(n)}\|)_{k\in\N}$ is nonincreasing. Hence we may assume in the
   following that each $\Phi^{(n)}$ is ordered in this manner. 

 Let $T$ denote the interval of positive values of $t$ for which $Z_1(t)<\infty$. 
  Let $t\in T$. Then, by letting $\eps$ tend to zero in (b)
  and diagonalizing, there is a sequence
  $(\kappa_n)_{n\in\N}$ of positive integers such that $\kappa_n=\kappa_n(t)$ grows at most
   subexponentially in $n$  
   and such that
    \begin{equation}\label{eqn:choiceofkappa}
     \lim_{n\to\infty} 
         \frac{\sum_{k>\kappa_n(t)} \|D\phi_k^{(n)}\|^{\tau}}%
         {\sum_{k\leq \kappa_n(t)} \|D\phi_k^{(n)}\|^{\tau}} = 0  \end{equation}
   for $\tau=t$. 
  Since the sequence $\|D\phi_k^{(n)}\|$ is nonincreasing in $k$, we automatically 
   obtain~\eqref{eqn:choiceofkappa} also for all $\tau>t$. Indeed, we can write
    \[ \frac{\sum_{k>\kappa_n(t)} \|D\phi_k^{(n)}\|^{\tau}}%
         {\sum_{k\leq \kappa_n(t)} \|D\phi_k^{(n)}\|^{\tau}} =
       \frac{\sum_{k>\kappa_n(t)} \left(\frac{\|D\phi_k^{(n)}\|}{\|D\phi_{\kappa_n(t)}\|}\right)^{\tau}}%
         {\sum_{k\leq \kappa_n(t)} \left(\frac{\|D\phi_k^{(n)}\|}{\|D\phi_{\kappa_n(t)}\|}\right)^{\tau}}. \]
  and the enumerator of the expression on the right-hand side is
   decreasing in $\tau$, while the denominator is increasing. 

 Letting $t$ tend towards the lower endpoint of $T$ and applying another diagonalization, 
   we can choose a sequence
   $\kappa_n$, independently of $t$, such that~\eqref{eqn:choiceofkappa} holds for all 
   $\tau\in T$.

  Likewise, we can choose a sequence $K_n$ such that 
    \[ \lim_{n\to\infty} \sum_{k\leq K_n} \|D\phi_k^{(n)}\|^t = \infty \]
   for all $t\in (0,d)\setminus T$, and such that $K_n$ grows at most
   subexponentially.

  Now set $I^{(n)}_F := \{k\in\N: k\leq \max(\kappa_n, K_n)\}$.  
   Then $\lP^{\Phi_F}(t) = \lP^{(\Phi)}(t)$ for all $t\in T$ by
   Proposition~\ref{p1na20111107}. On the other hand, we have
   $\lP^{\Phi_f}(t)=\infty=\lP^{\Phi}(t)$ for all $t\notin T$. 
\end{proof}

\

\begin{cor}\label{BowenM}
  Bowen's formula holds for all $\Phi \in \cM$.
\end{cor}
\begin{proof}
  By Theorem \ref{thm:bowensubexponential}, Bowen's formula holds 
   for the subsystem constructed in the previous observation. 
   So 
     \[ \HD(J(\Phi_F))\leq \HD(J(\Phi)) \leq B(\Phi)= B(\Phi_F) = 
        \HD(J(\Phi_F)).\qedhere \]
\end{proof}

 The requirements of Definition~\ref{d120120707} are technical and
  look somewhat awkward. Hence we shall now define a subclass
  of $\cM$ whose definition is easy to define and verify. 
  (In particular, this class contains all infinite autonomous 
   conformal iterated
  function systems.) This class will be used in
  Section~\ref{sec:rcifs} to establish Bowen's formula for random conformal
  iterated function systems. 

\

 \begin{defn}
  We say that an infinite non-autonomous system $\Phi$ is 
   \emph{evenly varying} if there is a sequence 
   $(\gamma_i)_{i\in\N}$ of positive real numbers and a constant $c\geq 1$
    such that 
    \[ \frac{\gamma_i}{c} \leq \|D\phi_i^{(n)}\| \leq
        c\cdot \gamma_i. \]

   The class of all evenly varying systems is denoted $\EV$. 
 \end{defn}

\

\begin{prop}
  $\EV\subset \cM$.
\end{prop}
\begin{proof}
 Let $\phi\in \EV$, and let $(\gamma_i)$ and $c$ be as above.
   By definition, $Z_1^{(n)}(t)$ is proportional to
   $\sum_{i} \gamma_i^t$ for all $n$ and $t$. Hence, for fixed  $t$,
   these sums are either 
   finite for all $n$ or infinite for all $n$, 
   depending on whether $(\gamma_i^t)$ is summable. 

  Suppose that $\sum_i \gamma_i^t < \infty$, and let $m,n\in\N$. Then 
    \[ \frac{\sum_{i=1}^{m} \|D\phi_i^{(n)}\|^t}{
        Z_1^{(n)}(t)} \geq 1 - \frac{c^2}{\gamma_1^t}\cdot \sum_{i=m+1}^{\infty} \gamma_i^t. 
    \]
   So this quotient tends to $1$ uniformly in $n$ as $m\to\infty$,
    which proves condition (b) of Definition~\ref{d120120707}.

   If $\sum_i\gamma_i^t = \infty$, then likewise
    \[ \sum_{i=1}^m \|D\phi_i^{(n)}\|^t \geq 
         \frac{1}{c} \sum_{i=1}^m \phi_i^t \]
   tends to infinity uniformly in $n$ as $\m\to\infty$, proving
   condition (c) of Definition~\ref{d120120707}.
\end{proof}

\begin{cor}\label{EVclass}
  Bowen's formula holds for all $\Phi\in\EV$. 
\end{cor}

\newcommand{\EP}{{\mathcal E}P}

\section{Random Conformal Iterated Function Systems}\label{sec:rcifs}

 In this section, we briefly discuss the connection of our results
  with the study of \emph{random conformal iterated function systems},
  as developed in \cite{RoyU}. Apart from the main results of the theory,
  the reader may find there an extended discussion of random systems, a number
  of examples, and a large literature. 
  Our goal is to reprove, generalize,
  and strengthen the main results regarding the
  Hausdorff dimensions of limit sets of these systems.

 A random CIFS (on a set $X\subset \R^d$ as in Section \ref{sec:definitions})
  consists of
  \begin{itemize}
    \item an invertible ergodic transformation 
       $T:\La\to\La$ preserving a probability measure $m$ on a 
       measure space $\La$;
    \item an alphabet $I$ (finite or countably infinite);
    \item for every $\lambda\in \La$, an autonomous conformal iterated function
       system 
      $\Psi_\lam=\{\phi_i^\lam:i\in I\}$ on such that the maps 
      $\lambda\mapsto\varphi_i^\lambda(x)$ are measurable for every $x\in X$. 
 \end{itemize} 

\sp\fr Thus for every $\lambda\in \La$, there is an associated non-autonomous system 
$$
\Phi_{\lambda}=\{\Phi_\lambda^{(n)}\}_{n=1}^\infty,
$$
where
$$
\Phi_{\lambda}^{(n)}=\{\phi_i^{T^n(\lambda)}:i\in I\}.
$$
Since $T$ is ergodic and $\lP^{\Phi_{\lambda}(t)} = \lP^{\Phi_{T(\lambda)}(t)}$,
  we have that $\lP^{\Phi_{\lambda}}(t)$ is constant for almost all
  $\lambda\in\La$. In  \cite{RoyU} this constant was denoted by $\EP(t)$ and was
  characterized by a formula which justifies its name -  the
  \emph{expected pressure}. 

 We are primarily interested in the limit set $J(\Phi_{\lam})$, and particularly
   its Hausdorff dimension, for almost every $\lambda\in \Lambda$. 
To state our main result, let us define, for every $i\in I$,
\[
\un M_i:=\ess\inf\{||D\phi_i^{\lam}||:\lam\in\La\}\quad\text{and}\quad
\ov M_i:=\ess\sup\{||D\phi_i^{\lam}||:\lam\in\La\}. \]

\begin{thm} 
  Let $\Phi$ be a random conformal iterated function system. If
   the alphabet $I$ is infinite, assume that
    \begin{equation}\label{120120709}
       \sup_{i\in I} \frac{\ov{M}_i}{\un{M}_i} < \infty. \end{equation}

  Then 
\[ \HD\(J\(\Phi_{\lambda}\)\)=\inf\{t\ge 0:{\mathcal E}P(t)\le 0\}, \]
   for almost every $\lambda\in\Lambda$,
   where $\EP(t)$, the \emph{expected pressure}, is the expected value of
   $\lP^{\Phi_{\lambda}}(t)$. 
\end{thm}

\begin{remark}[Remark 1]
 For infinite $I$, this reproves \cite[Theorem 3.18]{RoyU}
   in an entirely different way. For finite $I$, 
   \cite[Theorem 3.18]{RoyU} requires the assumption that
   $\un M_i>0$ for all $i\in I$, while we do not require
   any additional assumptions in this case.
\end{remark}
\begin{proof}
 
It suffices to prove that, for almost every $\lambda$,
  $\Phi_{\lambda}$ satisfies Bowen's formula. If $I$ is finite, this directly
  follows from Theorem \ref{thm:bowensubexponential}. 
  On the other hand, if $I$ is infinite, then it follows from the following
  observation.

 \begin{claim}
  If $I$ is infinite and $\Phi$ satisfies \eqref{120120709}, then
   $\Phi_{\lambda}\in \EV$ for almost every $\lambda\in\Lambda$. 
 \end{claim}
 \begin{subproof}
 Let $\Delta\geq 1$ be such that
    $\ov M_i \leq \Delta \un M_i$ for all $i$.
    Then, for all $i\in I$ and almost every $\lam\in\La$,
\[
\un M_i
\le ||D\phi_i^{T^n(\lam)}||
\le \ov M_i
\le\De\un M_i.
\]
So, putting $c:= \Delta$ and $\g_i:=\un M_i$ for all $i$, 
  we directly see that $\Phi_{\lam}\in\EV$. 
\end{subproof}
This concludes the proof of the Theorem.
\end{proof}

\

\section{Continuity of Pressure and Hausdorff
  Dimension}\label{continuity} 

In this section, we shall prove some results regarding the continuity 
  of pressure and Hausdorff dimension. Throughout the section, we will fix
  the sets $X\subset \R^d$ and $U\supset X$ involved in the definition
  of a non-autonomous conformal IFS, as well as the bounded distortion constant $K$
  and the contraction constant $\eta$ (together with
  the notion of when ``$m$ is large enough''
  for uniform contraction in Definition \ref{defn:ncifs}). 
  All systems considered
  in this section are assumed to satisfy Definition \ref{defn:ncifs} with this choice of 
  constants, so that the classes of systems studied here implicitly depend on these initial choices.

If $\Phi$ and $\Psi$ are two arbitrary autonomous conformal iterated
function systems (acting on the same set $X$) 
 with the same finite alphabet $I$, then $d(\Phi,\Psi)$, the
 distance between $\Phi$ and $\Psi$, is defined as follows:
\[
d(\Phi,\Psi)=\max\{\max\{||\phi_a-\psi_a||_\infty,||D\phi_a-D\psi_a||_\infty
\}:e\in I\}.
\]
Now, let $J=(I^{(k)})_{k=1}^\infty$ be an arbitrary sequence of finite
sets. If $\Phi=\{\Phi^{(n)}\}_{n=1}^\infty$ and
$\Psi=\{\Psi^{(n)}\}_{n=1}^\infty$ are arbitrary non-autonomous
conformal iterated function systems with the alphabets
$J=(I^{(k)})_{k=1}^\infty$, then
$d_u(\Phi,\Psi)$, the distance between $\Phi$ and $\Psi$, defined as follows:
\[
d_u(\Phi,\Psi)=\sup\lt\{d(\Phi^{(n)},\Psi^{(n)}):n\ge 1 \rt\} 
\le\max\{2,\diam(X)\}.
\]
The subscript $u$ indicates here that the distance is ``uniform'' with
respect to all $n\in\N$.

By $\K_u(X,J)$ we denote the space of all non-autonomous conformal iterated 
function systems with alphabet $J$, acting on $X$ and topologized by
the metric $d_u$, for which the following two conditions are satisfied:

\sp 
\begin{itemize}
\item[(Ka)] there is $\ka>0$ such that 
    $\inf\{|D\phi_a^{(n)}(x)|:n\ge 1,\, a\in I^{(n)},\,x\in X\}>\ka$.
\item[(Kb)] The family of functions $\{|D\phi_a^{(n)})|:n\ge 1, \,
  a\in I^{(n)}\}$ is uniformly continuous; i.e. 
$$
\aligned
\forall (\e>0)\ \  &\exists (\delta>0) \  \  \forall (n\ge 1) \  \
\forall (a\in I^{(n)})\\
&||x-y||<\delta \, \  \imp\, \
\lt||D\phi_a^{(n)}(x)|-|D\phi_a^{(n)}(y)|\rt|<\e.
\endaligned
$$
\end{itemize}

\begin{remark}
  Note that property (Ka) implies, in particular, that the systems we consider here
    are uniformly finite. (Infinite systems will be considered at the end of the section.) 

  Property (Kb) is a strengthening of the uniform distortion property of $\Phi$. 
    Observe that this property is, again, automatic in dimensions $\geq 2$. 
\end{remark}

\bobs\label{o120120921}
Observe that if all finite sets $I^{(k)}$, $k\ge 1$, are the same,
say equal to $I$, and $\Phi$ is an autonomous system with alphabet $I$
acting on $X$, then $\Phi\in \K_u(X,J)$, where $J=(I)_{n=1}^\infty$.
\eobs

 We now prove the following auxiliary result, asserting roughly 
that if the $d$ metric of two systems is small, then the
distance between all the corresponding contractions $\phi_\om$ and
$\psi_\om$ for all finite words $\om$ are also small.

\blem\label{l120120922}
If $\Phi,\Psi\in \K_u(X,J)$, then for all positive integers $m\le n$,
all $\omega\in I^{m,n}$, and all $x\in X$, we have 
$$
|\phi_{\omega}^{m,n}(x)-\psi_{\omega}^{m,n}(x)|
\le (1-\eta)^{-1}d_u(\Phi,\Psi).
$$
\elem
\begin{proof}
Let $\omega_*=\omega_m\omega_{m+1}\ld\omega_{n-1}$. We have
$$
\aligned
|\phi_{\omega}^{m,n}(x) & -\psi_{\omega}^{m,n}(x)| \\
&\le |\phi_{\omega^*}^{m,n-1}(\phi_{\omega_n}^{(n)}(x))-
          \phi_{\omega^*}^{m,n-1}(\psi_{\omega_n}^{(n)}(x))| +
     |\phi_{\omega^*}^{m,n-1}(\psi_{\omega_n}^{(n)}(x))-
          \psi_{\omega^*}^{m,n-1}(\psi_{\omega_n}^{(n)}(x))| \\
&\le \eta^{n-m}d(\Phi,\Psi) + 
      |\phi_{\omega^*}^{m,n-1}(x')-
      \psi_{\omega^*}^{m,n-1}(x')|,
\endaligned
$$
where $x':=\psi_{\tau_n}^{(n)}(x)$. So, proceeding inductively we obtain
\begin{equation} \label{eq-1star}
|\phi_{\omega}^{m,n}(x)-\psi_{\omega}^{m,n}(x)|
\le \sum_{i=0}^{n-m} \eta^{i} d(\Phi,\Psi) 
\le (1-\eta)^{-1}d(\Phi,\Psi)\,. \qedhere
\end{equation}
\end{proof}

 We are ready to prove a continuity property of the lower pressure
  function $\lP(t)$ in the class $\K_u(X,J)$.

\bthm\label{t1na25}
 For every $t\geq 0$, the lower pressure function 
    \[ \K_u(X,J) \to \R; \Psi\mapsto \lP^{\Psi}(t) \]
   is continuous. 
\ethm
\begin{proof}
 Let $\Phi\in \K_u(X,J)$. Note that, by definition of the distance $d_U$, there is
    $\kappa>0$ such that, if 
     $\Psi\in \K_u(X,J)$ is sufficiently close to $\Phi$, then 
    $\Phi$ satisfies (Ka) with this choice of $\kappa$. 

Fix $\e>0$, and let $\delta>0$ 
   according to condition (Kb) for the system $\Phi$. Take an
arbitrary system $\Psi\in \K(X,J)$ such that
$$
d_u(\Psi,\Phi)<(1-\eta)\min\{\e,\delta\}.
$$
We may assume that $\delta>0$ is chosen so small that $\Psi$ satisfies (Ka) for
   $\kappa$ as mentioned above.

Fix an arbitrary $n\ge 1$, $x,y\in X$ and $\om\in
I^{(n)}$. Using also (Kb) for $\Phi$ and Lemma~\ref{l120120922}, we then get
$$
\aligned
\lt|\log\frac{|D\psi_\om(x)|}{|D\phi_\om(x)|}\rt|
&=\lt|\sum_{j=1}^n\log\lt|D\psi_{\om_j}^{(j)}(\psi_{\sg^j(\om)}^{(j+1)}(x))
   \rt|-     
     \log\lt|D\phi_{\om_j}^{(j)}(\phi_{\sg^j(\om)}^{(j+1)}(x))\rt|\rt| \\
&\le
\sum_{j=1}^n\Bigg|\log\lt|D\psi_{\om_j}^{(j)}(\psi_{\sg^j(\om)}^{(j+1)}(x))
   \rt|-   
     \log\lt|D\phi_{\om_j}^{(j)}(\phi_{\sg^j(\om)}^{(j+1)}(x))\rt|\Bigg|
     \\
&=\sum_{j=1}^n\Bigg|\lt(\log\lt|D\psi_{\om_j}^{(j)}
        (\psi_{\sg^j(\om)}^{(j+1)} (x))   
   \rt|-\log\lt|D\phi_{\om_j}^{(j)}(\psi_{\sg^j(\om)}^{(j+1)}(x))\rt|\rt)+\\
& \qquad  +   \lt(
\log\lt|D\phi_{\om_j}^{(j)}(\psi_{\sg^j(\om)}^{(j+1)}(x))\rt| - 
   \log\lt|D\phi_{\om_j}^{(j)}(\phi_{\sg^j(\om)}^{(j+1)}(x))\rt|\rt)\Bigg|\\
&\le \sum_{j=1}^n\Bigg|\log\lt|D\psi_{\om_j}^{(j)}
        (\psi_{\sg^j(\om)}^{(j+1)} (x))   
   \rt|-\log\lt|D\phi_{\om_j}^{(j)}(\psi_{\sg^j(\om)}^{(j+1)}(x))
   \rt|\Bigg|+\\ 
&\qquad  +   \sum_{j=1}^n\Bigg| 
\log\lt|D\phi_{\om_j}^{(j)}(\psi_{\sg^j(\om)}^{(j+1)}(x))\rt| - 
   \log\lt|D\phi_{\om_j}^{(j)}(\phi_{\sg^j(\om)}^{(j+1)}(x))\rt|\Bigg|\\
&\le \ka^{-1}\sum_{j=1}^n
     \Big|\lt|D\psi_{\om_j}^{(j)}(\psi_{\sg^j(\om)}^{(j+1)}(x))\rt| - 
     \lt|D\phi_{\om_j}^{(j)}(\psi_{\sg^j(\om)}^{(j+1)}(x))\rt|\Big|+\\
&\qquad  + \ka^{-1}\sum_{j=1}^n 
      \Big|\lt|D\phi_{\om_j}^{(j)}(\psi_{\sg^j(\om)}^{(j+1)}(x))\rt| - 
      \lt|D\phi_{\om_j}^{(j)}(\phi_{\sg^j(\om)}^{(j+1)}(x))\rt|\Big|\\
&\le \ka^{-1}nd_u(\Psi,\Phi)+\ka^{-1}n\e \\
&\le  2\ka^{-1}\e n.
\endaligned
$$
Therefore,
$$
K^{-1}\exp\(- 4\ka^{-1}\e n\) 
\le \frac{\|D\psi_\om\|}{\|D\phi_\om\|}
\le K\exp\(4\ka^{-1}\e n\).
$$
Hence,
$$
K^{-t}\exp\(- 4\ka^{-1}t\e n\)
\le  \frac{Z_n^{\Psi}(t)}{Z_n^{\Psi}(t)}
\le K^t\exp\(4\ka^{-1}t\e n\).
$$
So,
$$
|\un\P_\Psi(t)-\un\P_\Phi(t)|<4\ka^{-1}t\e.
$$
Since $\eps>0$ was arbitary, we are done.
\end{proof}

Since the pressure function $[0,+\infty)\ni t\mapsto
\lP^\Psi(t)$ is strictly decreasing, and, by condition (Ka), our
system is uniformly bounded, along with
Theorem~\ref{thm:bowensubexponential}, this theorem gives
the following.

\bcor\label{c2na25}
The Hausdorff dimension function
$\K_u(X,J)\to [0,d]; \Phi\mapsto\HD(J(\Phi))$ is continuous.
\ecor

Our consideration in this section up to now applied only to uniformly finite
  systems. We shall now deal with much more
general finite and truly infinite non-autonomous systems.

\bdfn\label{d120120707B}
Assume that for every $n\ge 1$, $I^{(n)}$ is an initial segment of $\N$
(i.e., either $I^{(n)}= \{1,\dots,m\}$ for some $m$, or  $I^{(n)}=\N$). 
 Let $J=(I^{(n)})_{j=1}^\infty$, and let
  $\Phi$ be a nonautonomous system with alphabet $J$. 

Suppose that $\Phi$ satisfies condition (Kb) -- i.e., the
  derivatives $\phi_a^{(n)}$ are uniformly continuous -- and that
  \[ \inf_{n\in\N} \|D\phi_a^{(n)}\| > 0 \]
   for all $a$. 

Suppose furthermore
  that, for every $t\in [0,d]$  and all $\delta>0$, there is $M>0$ such that
 \begin{enumerate}
   \item the sums $Z_1^{(n)}(t)$ are either infinite for all $n$ or finite for all $n$;
   \item if $Z_1(t)<\infty$, then, for all $n\geq 1$, 
        \[ \sum_{k\leq M} \|D\phi_k^{(n)}\|^t \geq (1-\delta)\cdot Z_1^{(n)}(t); \]
   \item if $Z_1(t)=\infty$, then, for all $n\geq 1$, then
       \[ \sum_{k\leq M} \|D\phi_k^{(n)}\|^t \geq 1/\delta. \]
 \end{enumerate}

  Then we say that
  $\Phi$ belongs to the class $\cM_+$.
\edfn

Since $\cM_+\sbt \cM$, we immediately obtain the following
 fact by virtue of Corollary~\ref{BowenM}.

\bthm\label{BowenM2}
Bowen's formula holds for all systems $\Phi \in \cM_+$.
\ethm

For any systems $\Phi, \Psi\in\cM_+$ define
$$
d_p(\Phi,\Psi)=\sum_{n=1}^\infty 
  2^{-n}d\(\Psi^{(n)},\Phi^{(n)}\).
$$
The number $d_p(\Phi,\Psi)$ is obviously a metric on the space
$\cM$. In contrast to $d_u(\Phi,\Psi)$ we refer to it as the pointwise
distance between the systems $\Phi$ and $\Psi$. 

\

\bdfn\label{b0na27}
Let $\Phi\in\cM_+$ and $\(\Phi_n\)_{n=1}^\infty\sbt\cM_+$. We say that
the sequence $\(\Phi_n\)_1^\infty$ $\om$--converges to $\Phi$ if 
\begin{itemize}
\item[(a)]
$$
\lim_{n\to\infty}d_p(\Phi_n,\Phi)=0
$$
and

\sp\item[(b)]
 For every $t\in [0,d]$ and all $\delta>0$, the number $M$ in the definition of the
   class $\cM_+$ can be chosen independently of $n\geq 1$. 
\end{itemize}
\edfn

\

\fr Note that in the case of autonomous sequences, $\lam$--convergence
from \cite{royurblambda} implies $\om$--convergence. Our conditions here
are weaker! We shall prove the following.

\

\bthm\label{t1na27}
If $\Phi\in\cM_+$, $\(\Phi_n\)_1^\infty\sbt\cM_+$ and
$\(\Phi_n\)_1^\infty$ $\om$-converges to $\Phi$, then
$$
\lim_{n\to\infty}\lP_{\Phi_n}(t)=\lP_{\Phi}(t)
$$
for all $t\ge 0$.
\ethm
\begin{proof}
  For $M\in\N$, 
   let $\Phi(M)$ and $\Phi_n(M)$ denote the finite systems obtained by truncating 
    all alphabets at symbol $M$. 

   By the definition of class $\cM_+$ and Proposition~\ref{p1na20111107},
    we see that, for every $t$, 
      \[ \lP^{\Phi(M)}(t) \to \lP^{\Phi}(t) \]
     as $M\to\infty$, and similarly for $\Phi_n$. Furthermore, the convergence is
    uniform in $n$. 

   Furthermore, the system $\Phi(M)$ belongs to the class
     $\K_u(X, \tilde{J})$ (where $\tilde{J}$ is the alphabet of the system $\Phi(M)$), and
     $\Phi_n(M)$ converges to $\Phi(M)$ in the $d_u$-metric as $n\to\infty$. 

    The claim follows.
\end{proof}

 Recall that, for every non-auonomous system $\Phi$, the pressure function
$t\mapsto \P_{\Phi}(t)$ is strictly decreasing throughout its domain
of finiteness. Thus we obtain the following
  immediate consequence of Theorem~\ref{t1na27} and
 Theorem~\ref{BowenM2} (Bowen's formula).

\bthm\label{t1na29}
If $\Phi\in\cM_+$, $\(\Phi_n\)_1^\infty\sbt\cM_+$ and
$\(\Phi_n\)_1^\infty$ $\om$--converges to $\Phi$, then
$$
\lim_{n\to\infty}\HD\(J_{\Phi_n}\)=\HD\(J(\Phi)\).
$$
\ethm

\section{Appendix: Unbounded orbits in Julia sets and in limit sets of infinite autonomous IFS} \label{sec:applications}

Let $\Phi=\{\phi_n:n\in\N\}$ be an infinite (autonomous) conformal iterated
function system. (So $I^{(n)}=\N$ for all $n$, and $I^{\infty}=\N^{\N}$.) We are interested in the set of
 points $x\in J(\Phi)$ that have dense orbits.

More precisely, let $\sg:\N^\N\to\N^\N$ be the shift map, i.e.
$$
\sg(\om)_n=\om_{n+1}.
$$
We define
$$
I^{\infty}_{\dense}=\lt\{\om\in\N^\N:\ov{\{\sg^n(\om):n\ge 0\}}=\N^\N\rt\},
$$
i.e $I^{\infty}_{\dense}$ consists of all points in $\N^\N$ whose forward
orbit under the shift map is dense in $\N^\N$. Recall the definition of the
 projection map $\pi_{\Phi}:I^{\infty}\to J(\Phi)$ from Definition \ref{deflem:projection}. We set 
$$
J_\dense(\Phi)=\pi_\Phi\(I^{\infty}_{\dense}\).
$$
Since $J(\Phi)=\pi_\Phi\(I^{\infty}\)$ and the map $\pi_\phi$ is continuous, we thus get that
$$
\ov{\{\pi_\Phi(\sg^n(\om)):n\ge 0\}}=\ov{J(\Phi)}.
$$

\begin{thm}\label{denseIFS}
If $\Phi$ is an infinite autonomous iterated function system, then
$$
\HD(J_\dense(\Phi))=\HD(J(\Phi)).
$$
\end{thm}
\begin{remark}
This result was previously known in the case when the system $\Phi$
is strongly regular, or at least if it is regular and its entropy is finite (see Corollary~4.4.6 in
\cite{mauldinurbanskiGDMS}), but not generally. We note that our proof also works for finite systems,
but here the result is well known; see
\cite[Corollary~4.4.6]{mauldinurbanskiGDMS}
\end{remark}
\begin{proof}
 Let $\Psi$ be the non-autonomous subsystem of $\Phi$ defined by the index sets
    \[ I^{(n)}_f := \{1,\dots, n\}. \]
 It follows from Proposition \ref{prop:approximation} that $B(\Psi)=B(\Phi)$. 

Now let $\{\om^{k}\}_{k\in\N}$ be an enumeration of all
finite words; i.e., an enumeration of the (countable) set
  $\bigcup_{j}I^j$. Given an arbitrary increasing sequence
$(n_k)_1^\infty$, we define a system $\Theta$ by 
\[ \Theta^{(n)} := \begin{cases}
                               \Psi^{(n)} &\text{if $n\neq n_k$ for all $k\in\N$}; \\
      \{\phi_{\om^{k}}\} & \text{if $n = n_k$ for some $k\in\N$}. \end{cases} \]                               
It is clear from this definition that
\beq\label{1na39}
J(\Theta)\sbt J_\dense(\Phi).
\eeq

Furthermore, if the sequence $(n_k)$ is chosen to grow sufficiently rapidly, then
\begin{equation}\label{eqn:equalityofpressures}
   P^{\Theta}(t)=P^{\Psi}(t) 
\end{equation}
for all $t \geq 0$. Indeed, the systems 
   $\Theta$ and $\Psi$ agree except at the times $n_k$, where the contribution to
   $Z_n^{\Psi}$ (with $n\geq n_k$)
    is
    \begin{equation}\label{eqn:ZnPsi} Z_1^{(n_k)}(t) = \sum_{m\in I_f^{(n_k)}} \|D\phi_m\|^t =
            \sum_{m=1}^{n_k} \|D\phi_m\|^t \leq n_k, \end{equation}
   while the contribution to $Z_n^{\Theta}$ is precisely
    $\|D\phi_{\omega^k}\|^t$. Also observe that the sum $Z_1^{(n_k)}(t)$ in~\eqref{eqn:ZnPsi}
    is bounded below by
    $\|D\phi_1\|^d$. 

So let us suppose that $(n_k)$ is chosen such that
    $\log n_k \geq \max(k, |\log \|D\phi_{\omega^k}\|\,|)$. 
    Let $n\in\N$ be large, and let $k$ be maximal such that
    $n_k\leq n$. Then 
  \begin{align*}
    |\log Z_n^{\Theta}(t) - \log Z_n^{\Psi}(t)| &\leq
     2k \log K + \sum_{j=1}^k \left|\log Z_1^{(n_j)}(t)\right| 
        + \sum_{j=1}^k t\left|\log\|D\phi_{\omega^j}\|\,\right|\\
    &\leq 2k \log K + k\cdot \log(n_k) +  k\cdot d \cdot \log(n_k) = o(n),
   \end{align*}
  implying~\eqref{eqn:equalityofpressures}. 
By Theorem \ref{thm:bowensubexponential}, it follows that
   \[ \HD(J(\Phi)) \geq \HD(J(\Theta)) = B(\Theta) = B(\Psi) = B(\Phi) \geq \HD(J(\Phi)), \]
 as claimed. 
\end{proof}

We can apply the same principle to subsets of Julia sets of certain holomorphic
 functions known as \emph{Ahlfors islands maps}. 
  For definitions, we refer to \cite{hypdim}. Here we note only that
 every non-constant, non-linear meromorphic function is an Ahlfors islands map; hence the following result
 implies
  Theorem~\ref{thm:meromorphic}.

\begin{thm}[Subsets of Julia sets]\label{juliadense}
Let $S$ be a compact Riemann surface, let $U\subset S$ be open and
nonempty, and let $f:U\to S$ be a nonelementary 
Ahlfors islands map in the sense of Epstein \cite[Definition~2.1]{hypdim}.  

Set 
  \[ J_{r,\dense}(f) := \{z\in J_r(f): \cl{\bigcup_{j\geq 0} f^j(z)}=J(f)\}, \]
  where $J_r(f)$ is the radial Julia set of $f$ 
  \cite[Definition~2.5]{hypdim}.
   Then 
      \[ \hdim(J_{r,\dense}(f)) = \hdim(J_r(f)) = \hypdim(f), \]
     where $\hypdim(f)$ denotes the hyperbolic dimension of $f$.
 \end{thm}
 \begin{proof}
The second equality was proved in \cite{hypdim}. It follows from the proof that, for every 
 $h\in [0,\hypdim(f) )$, there is a finite conformal iterated function system $\Psi=\{\psi_i:i\in I\}$ 
 with the following properties.

\begin{enumerate}
\item  $X$, the domain of the system $\Psi$, is a closed topological disk with analytic boundary.
\item Each map $\psi_i$, $i\in I$, is a holomorphic inverse branch of
  some iterate of $f$, and extends to a conformal map defined on some larger disk $U\supset X$, independent of
     $i$.
\item $\HD(J(\Psi))\ge h$.
\item For every $z\in J(f)$ and every $\eps > 0$ there exists an
  analytic inverse branch
  $\psi_{z,\eps}:U\to S$ of some iterate of $f$ 
  such that 
  $\psi_{z,\eps}(U)\sbt B(z,\eps)$. (To measure distances, we fix 
  some metric on $S$; e.g.\ a conformal metric of constant curvature.)
\end{enumerate}

Now choose a sequence $(z_k)_{k\in\N}$ in $J(f)$ that is dense in $J(f)$. Since $J(f)$ is uncountable, we may assume that
  no $z_k$ belongs to the orbit of a critical point. 
  Now $\interior(X)$ intersects the Julia set of $f$~-- indeed, $J(\Psi)\subset J(f)$ by definition~--
  and hence the family $f^n|_{\interior(X)}$ omits at most finitely many values by the
   blowing-up property of the Julia set. Hence we can assume without loss of
   generality that each $z_k$ has an iterated preimage in $\interior(X)$.
  That is, for each $k$, there is some $m_k$ and an iterated preimage $\tilde{z}_k\in f^{-m_k}(z_k)$ such that
    $\tilde{z}_k\in \interior(X)$. 
   Let $\eps=\eps_k\leq 1/k$ be so small that
    the branch $\theta_k$ of $f^{-m_k}$ that takes $z_k$ to $\tilde{z}_k$ is 
    defined on $B(z_k,\eps_k)$ and takes values in 
    $\interior(X)$ there. 

 Define
   \[ \phi_k := \theta_k\circ \psi_{z_k,\eps_k},  \]
 let $(n_k)$ be a sufficiently growing sequence of positive integers, and set
    \[ \Phi^{(n)} := \begin{cases}
                            \Psi &\text{if $n\neq n_k$ for all $k\in\N$}; \\
                              \{\phi_k\} &\text{if $n=n_k$ for some $k\in\N$}.\end{cases} \]

  By definition, the $f$-orbit of any point in $J(\Phi)$ passes through 
    $B(z_k,1/k)$ for all $k$, and hence is dense in $J(\Phi)$. 
    We also note that $J(\Phi)\subset J_r(f)$ by the definition of the 
    radial Julia set. 

The non-autonomous system $\Phi$ is uniformly finite,
  hence $\HD(J(\Phi))=B(\Phi)$ by
   Theorem \ref{thm:bowensubexponential} (or Corollary
   \ref{cor:weaklybalancedbowen}).
   If the sequence $(n_k)$ grows sufficiently quickly, then  (as in the proof of the
   preceding theorem)
    $B(\Phi)=B(\Psi)=\HD(J(\Psi))\geq h$. Since $h<\hypdim(f)$ was
    arbitrary, we are done. 
 \end{proof}


\begin{thebibliography}{10}

\bibitem{bowen}
R. Bowen, 
\emph{Hausdorff dimension of quasi-circles, Publ. Math}. IHES,\textbf{50}
(1980), 11-25.

\bibitem{four_authors}
Y. Dai, Z. Weng, L. Xi, and Y. Xiong,
Quasisymmetrically minimal Moran sets and Hausdorff dimension, Annales
Academiae Scientiarum Fennicae Mathematica, 36 (2011), 139–151. 

\bibitem{DUsul}
M. Denker and Mariusz Urba\'nski, \emph{On Sullivan's conformal measures for
rational maps of the Riemann sphere}, Nonlinearity \textbf{4} (1991), 365-384.

\bibitem{hensley}
D. Hensley,\emph{Continued fractions}, World Scientific (2006).

\bibitem{holland}
M. Holland, Y. Zhang,
Dimension results for inhomogeneous Moran set constructions, Preprint 2012.

\bibitem{falconerfractal}
 K. Falconer, \emph{Fractal Geometry}, John Wiley \& Sons, 1990.

\bibitem{jordanrams}
T.\ Jordan and M.\ Rams, 
  \emph{Increasing digit subsystems of infinite iterated function systems}, 
   Proc. Amer. Math. Soc. \textbf{140} (2012), 1267-1279.

\bibitem{luczak}
Tomasz {\L}uczak, \emph{On the fractional dimension of sets of
  continued fractions}, 
 Mathematika \textbf{44} (1997), no. 1, 50–53. 

\bibitem{mulondon}
D. Mauldin and Mariusz Urba\'nski, \emph{Dimensions and measures in
  infinite iterated 
function systems}, Proc. London Math. Soc. (3) \textbf{73} (1996) 105-154.

\bibitem{mauldinurbanskiGDMS}
D. Mauldin and Mariusz Urba\'nski, \emph{Graph Directed Markov
Systems: Geometry and Dynamics of Limit Sets}, Cambridge University Press
(2003)

\bibitem{moran}
P.A.P. Moran,  \emph{Additive functions of intervals and Hausdorff
measure}, Proc. Cambridge, Philos. Soc. \textbf{42} (1946), 15-23.

\bibitem{PUbook}
F. Przytycki and Mariusz Urba\'nski, \emph{Conformal Fractals —
Ergodic Theory Methods}, Cambridge University Press (2010). 

\bibitem{hypdim}
 L.\ Rempe, \emph{Hyperbolic dimension and radial Julia sets of
   transcendental functions}, Proc.\ Amer.\ Math.\ Soc.\ \textbf{137} (2009),
   1411--1420.

\bibitem{royurblambda}
M. Roy and Mariusz Urba\'nski, \emph{Regularity properties of
  Hausdorff dimension 
in conformal infinite IFS}, Ergodic Th. \& Dynam. Sys.\textbf{25} (2005),
1961-1983.

\bibitem{royurb2}
M. Roy, H. Sumi and Mariusz Urba\'nski, \emph{Analytic Families of
  Holomorphic IFS}, 
Nonlinearity \textbf{21} (2008), 2255-2279. 

\bibitem{royurb3}
M. Roy, H. Sumi and Mariusz Urba\'nski, \emph{Lambda-topology
  vs. pointwise topology}, 
Ergodic Th. and Dynam. Sys., \textbf{29} (2009), 685-713.

\bibitem{RoyU}
Mario Roy and Mariusz Urba\'nski, \emph{Random graph directed Markov
systems}, Discrete Contin.\ Dyn.\ Syst.\ \textbf{30} (2011), no. 1, 261-298. 

\bibitem{UZmichigan}
 Mariusz Urba\'nski and A. Zdunik, \emph{The finer geometry and dynamics of
exponential family}, Michigan Math. J. \textbf{51} (2003), 227-250.

\bibitem{moransetsclasses}
Zhiying Wen,  \emph{Moran sets and Moran classes},
   Chinese Sci. Bull. \textbf{46} (2001), no. 22, 1849–-1856. 

\end{thebibliography}
\end{document}